\documentclass[a4paper,onecolumn, superscriptaddress, 10pt]{compositionalityarticle}
\pdfoutput=1
\usepackage[utf8]{inputenc}
\usepackage[english]{babel}
\usepackage[T1]{fontenc}
\usepackage{amsmath}
\usepackage{hyperref}
\usepackage{tikz}
\usepackage{lipsum}
\usepackage[numbers,sort&compress]{natbib}
\usepackage{amsmath,amssymb,amsthm,mathtools}
\usepackage{graphicx,subfigure}

\newcommand{\authorsforheader}{Bumpus, Libkind, Lopez Garcia, Sorkatti, Tenka}
\newcommand{\compositionalityIssue}{4}
\newcommand{\compositionalityVolume}{7}
\newcommand{\paperdoinumber}{10.46298/compositionality-\compositionalityVolume-\compositionalityIssue}
\newcommand{\paperdoi}{https://doi.org/\paperdoinumber}
\newcommand{\receiveddate}{2023-03-06}
\newcommand{\accepteddate}{2025-04-21}

\usepackage{thmtools}
\declaretheorem[within=section]{theorem}

\declaretheorem[sibling=theorem]{lemma}
\declaretheorem[sibling=theorem,style=definition]{definition}
\declaretheorem[sibling=theorem,style=remark]{remark}
\declaretheorem[sibling=theorem,style=remark]{example}

\DeclareMathOperator{\openpetri}{\mathbf{OPetri}}
\DeclareMathOperator{\mope}{\mathbf{MOPetri}}
\DeclareMathOperator{\finset}{\mathbf{FinSet}}

\newcommand{\bt}{\beta}

\DeclareMathOperator{\bn}{\cc{B}\mathbb{N}}

\newcommand{\rest}{|}

\newcommand{\cp}{\mathbb{;}}
\newcommand{\id}{\textrm{id}}
\usepackage{pgf,tikz,environ}
\usepackage{quiver}
\usetikzlibrary{arrows}
\usetikzlibrary{cd}
\usetikzlibrary{positioning}
\usetikzlibrary{fadings}
\usetikzlibrary{decorations.pathreplacing}
\usetikzlibrary{decorations.pathmorphing}
\tikzset{snake it/.style={decorate, decoration=snake}}

\newcommand{\defined}[1]{\textbf{#1}}

\usepackage{adjustbox}

\usepackage{tikz}
\usepackage{tikz-cd}
\usetikzlibrary{backgrounds,circuits,circuits.ee.IEC,shapes,fit,matrix}

\tikzstyle{simple}=[-,line width=2.000]
\tikzstyle{arrow}=[-,postaction={decorate},decoration={markings,mark=at position .5 with {\arrow{>}}},line width=1.100]
\pgfdeclarelayer{background}
\pgfdeclarelayer{edgelayer}
\pgfdeclarelayer{nodelayer}
\pgfsetlayers{background,edgelayer,nodelayer,main}

\tikzstyle{none}=[inner sep=0pt]

\definecolor{lblue}{rgb}{0,250,255}
\tikzstyle{species}=[circle,fill=yellow,draw=black,scale=1.15]
\tikzstyle{transition}=[rectangle,fill=lblue,draw=black,scale=1.15]
\tikzstyle{inarrow}=[->, >=stealth, shorten >=.03cm,line width=1.5]
\tikzstyle{empty}=[circle,fill=none, draw=none]
\tikzstyle{inputdot}=[circle,fill=purple,draw=purple, scale=.25]
\tikzstyle{inputarrow}=[->,draw=purple, shorten >=.05cm]
\tikzstyle{simple}=[-,draw=purple,line width=1.000]

\newcommand{\pr}{^\prime}

\newcommand{\cc}[1]{\mathbf{#1}}

\DeclareMathOperator{\Ea}{\mathcal{E}}  
  
\DeclareMathOperator{\Ga}{\mathcal{G}}

\DeclareMathOperator{\Ma}{\mathcal{M}}  
  \DeclareMathOperator{\Nb}{\mathbb{N}}
  
\DeclareMathOperator{\Pa}{\mathcal{P}}  
\DeclareMathOperator{\Qa}{\mathcal{Q}}  
\DeclareMathOperator{\Ra}{\mathcal{R}}

\newcommand{\Nn}{\mathbb{N}}

\begin{document}

\title{Additive Invariants of Open Petri Nets}
\date{}

\author{Benjamin Merlin Bumpus}
\email{benjamin.merlin.bumpus(at)gmail.com}
\orcid{0000-0002-8686-2319}
\thanks{While preparing this article, this author also received funding by the European Research Council (ERC) under the European Union's Horizon 2020 research and innovation programme (grant agreement No 803421, ReduceSearch).}
\affiliation{Department of Computer and Information Science and Engineering, University of Florida, Gainesville, USA.}

\author{Sophie Libkind}
\email{sophie(at)topos.institute}
\orcid{0009-0008-1354-7589}
\affiliation{Topos Institute, Berkeley, USA.}

\author{Jordy Lopez Garcia}
\email{jordy.lopez(at)tamu.edu}
\orcid{ 0000-0001-6022-5883}
\affiliation{Texas A\&M University, Department of Mathematics, College Station, USA.}

\author{Layla Sorkatti}
\email{layla.sorkatti(at)bath.edu}
\orcid{0000-0001-6507-3858}
\affiliation{Southern Illinois University, School of Mathematical and Statistical Sciences, Carbondale, USA; 
University of Khartoum and Al-Neelain University, Department of Pure Mathematics, Khartoum, Sudan.
}
\thanks{While preparing this article, this author also received funding by the Southern Illinois University, Simon Foundation and AMS-Simons Travel Grant.}

\author{Samuel Tenka}
\email{samtenka(at)mit.edu}
\affiliation{MIT, Computer Science and Artificial Intelligence Laboratory (CSAIL), Cambridge, USA.
}
\maketitle

\begin{abstract}
We classify all additive invariants of open Petri nets: these are $\mathbb{N}$-valued invariants which are additive with respect to sequential and parallel composition of open Petri nets. In particular, we prove two classification theorems: one for open Petri nets and one for monically open Petri nets (i.e.\ open Petri nets whose interfaces are specified by monic maps). Our results can be summarized as follows. The additive invariants of open Petri nets are completely determined by their values on a particular class of single-transition Petri nets. However, for monically open Petri nets, the additive invariants are determined by their values on transitionless Petri nets and all single-transition Petri nets. Our results confirm a conjecture of John Baez (stated during the AMS' 2022 Mathematical Research Communities workshop). 
\end{abstract}

\section{Introduction} 
\newcommand{\hydrogen}{\textrm{H}}
\newcommand{\oxygen}{\textrm{O}}

A Petri net is a directed bipartite graph with bipartition $(S,T)$ consisting of a set $S$ of \emph{species} and a set $T$ of \emph{transitions}. They were invented by Carl Petri in 1939~\cite{petri-petrinet} as a graphical representation of a set $S$ of chemical compounds that can be combined by way of a set of reactions $T$, into new compounds.
For example, consider the following Petri net representing the electrolysis
$
  \mathcal{E} \colon 2\hydrogen_2\oxygen \to 2\hydrogen_2 + \oxygen_2
$ of water to make oxygen and hydrogen gas. It consists of only one transition --- namely $\mathcal{E}$ --- and three species: $\hydrogen_2\oxygen$, $\hydrogen_2$ and $\oxygen_2$ as drawn below.

\begin{center}
    \begin{tikzpicture}
\begin{pgfonlayer}{nodelayer}
	\node [style=species ] (A) at (-4.5, 0) {$\hydrogen_2\oxygen$};
	\node [style=species] (B) at (-1, 1) {$\hydrogen_2$};
	\node [style=species] (C) at (-1,-1) {$\oxygen_{2}$};
    \node [style=transition, fill=green!40, inner sep= 2mm] (aa) at (-2.5, 0) {$\varepsilon$}; 
\end{pgfonlayer}

\begin{pgfonlayer}{edgelayer}
	\draw [style=inarrow, bend right=15, looseness=1.00] (A) to (aa);
	\draw [style=inarrow, bend left= 15, looseness=1.00] (A) to (aa);
	\draw [style=inarrow, bend right=15, looseness=1.00] (aa) to (B);
	\draw [style=inarrow, bend left=15, looseness=1.00] (aa) to (B);
	\draw [style=inarrow] (aa) to (C);
\end{pgfonlayer}
\end{tikzpicture}
\end{center}

In general, the applications of Petri nets need not be confined to chemistry. Indeed, they can represent many kinds of processes (concurrent, asynchronous, distributed, parallel, nondeterministic and stochastic, to name a few) in which some entities (the species) undergo transformations (transitions) in order to be converted into other kinds of entities~\cite{baez2018quantum,haas2006stochastic,jensen2009coloured,koch2010petri,koch2010modeling,marsan1998modelling, petri-net-theory-book, wilkinson2018stochastic}.

Formally, a \defined{Petri net} $P$ with \textit{finitely many} \defined{species} $S$ and \defined{transitions} $T$ is a graph with source and target maps $s,t \colon T \to \mathbb{N}^S$, where $\mathbb{N}^S$ is the free commutative monoid on $S$.  We denote a Petri net $P$ by the quadruple $(S, T, s, t)$. In the example above, the source and target maps $s$ and $t$ are defined by:
\begin{align*}
    s(\Ea)(\hydrogen_2 \oxygen) = 2, \quad  s(\Ea)(\hydrogen_2) = 0, \quad s(\Ea)(\oxygen_2) =0,\\
    t(\Ea)(\hydrogen_2 \oxygen) = 0, \quad  t(\Ea)(\hydrogen_2) = 2, \quad t(\Ea)(\oxygen_2) = 1.
\end{align*} 
For a transition $\tau \in T$ and species 
$\sigma \in S$, the quantity \(s(\tau)(\sigma)\) represents the number of \defined{input arcs} which emanate from the species \(\sigma\) to the transition \(\tau\). Similarly, the quantity \(t(\tau)(\sigma)\) represents the number of \defined{output arcs} which emanate from the transition \(\tau\) to the species \(\sigma\).
We say that $\sigma \in S$ is an \defined{input species} (resp. \defined{output species}) of the transition $\tau \in T$ if $s(\tau)(\sigma) > 0$ (resp. $t(\tau)(\sigma) > 0$). Accordingly, the total number of input arcs into and output arcs out of a transition \(\tau\) are given by: 
\begin{equation}\label{eq:input-arcs}
    \sum_{\sigma \in S} s(\tau)(\sigma) 
\end{equation}
and 
\begin{equation}\label{eq:output-arcs}
    \sum_{\sigma \in S} t(\tau)(\sigma).
\end{equation}
In the example of electrolysis, the transition $\Ea$ has two input arcs and three output arcs. The molecule $\hydrogen_2 \oxygen$ is its input species and the molecules $\hydrogen_2$ and $\oxygen_2$ are its output species.

\subsection{Invariants of composable Petri nets}

The study of composable Petri nets was introduced by Baldan, Corradini, Ehrig, Heckel, and K\"{o}nig~\cite{baldan_corradini_ehrig_heckel_2005,baldan2001compositional}. Baez and Pollard~\cite{reaction-networks-pollard-baez} used the formalism of decorated cospans to introduce tensoring of open Petri nets. In the same vein as our decomposition lemmas that appear in Section~\ref{sec:decomposition}, Nielsen, Priese and Sassone~\cite{nielsen1995characterizing} showed that a finite Petri net is built from single-place and single-transition Petri nets via a collection of operations on nets known as \textit{combinators}. Gadduci and Heckel~\cite{gadducci1997inductive} presented a theorem (referred to as \textit{the Kindred theorem}) for the decomposition of graphs into fundamental components.

Many fields have vast databases that record many kinds of chemical reactions and their associated Petri nets~\cite{keserHu2014anthropogenic,lawson2014making, reedexpanded, van2012substances}. These are studied empirically by computationally seeking patterns, \textit{motifs}~\cite{tyson2010functional}, and \textit{numerical invariants} that arise in this data. Often due to the sheer size of the Petri nets in such databases, it is convenient to think of a large Petri net as being built out of smaller constituent nets which are ``glued'' together to form the entire net. This compositional structure of Petri nets is particularly useful when one wishes to study structural measurements that are \textit{isomorphism-invariant} and which are \textit{compositional} in the sense that they behave nicely with respect to this kind of gluing. In this article, as is customary in many areas of mathematics (cf. algebraic topology, graph theory etc.) we will adopt the term \defined{invariant} as a more concise substitute for the term
``isomorphism-invariant measurement''. 

The literature on Petri nets contains other uses of the term invariant (e.g. P- and T- invariants) which are topological properties which do not depend on the initial marking of a Petri net. This usage is very similar to ours; however, in this article we are not concerned with markings and instead we are interested in invariants which satisfy the following two requirements:
\begin{enumerate}
    \item For a Petri net built out of smaller parts, the overall invariant value should be determinable solely from the values on the pieces.

    \item It is \textit{computationally cheap} to compute the invariant on the whole Petri net from its values on the pieces.
\end{enumerate}
The first requirement is in a similar vein to that of the work of Baldan, Corradini, Ehrig, and Heckel~\cite{baldan_corradini_ehrig_heckel_2005} (which studies compositional invariants different from the ones studied in this article). On the other hand, the second requirement is born from the overarching question which, sitting in the background, motivates this paper; namely: ``\textit{which structural invariants of very large, real-world Petri nets can be used to synthetically generate Petri nets with comparable structural statistics?}'' Although this question is beyond the scope of the present paper, it serves as a motivation to determine \textit{isomorphism-invariant} measurements which are both \textit{compositional} and \textit{easily computable}, even on truly vast Petri nets.

There are several software systems for graphical modeling. In particular,  \href{https://github.com/AlgebraicJulia/AlgebraicPetri.jl}{AlgebraicPetri.jl} supports the definition and composition of open Petri nets.  More recently, Catcolab~\cite{catcolab} implements a more dynamic modeling setting that includes \defined{motif analysis}~\cite{motifs} for stock and flow diagrams, which is a strategy that could be generalized to identify the atoms of a Petri net. Together, these tools demonstrate the feasibility of implementing additive invariants for open Petri nets. Such an implementation would allow us to generate large synthetic Petri nets whose invariants satisfy certain constraints.

To that end, this paper focuses on one of the first, obvious requirements that one might put on an invariant so as to render its compositional computation efficient, that is, \textit{additivity}. By this we mean that we are interested in invariants $F$ which can be computed on any large Petri net $P$ as a sum $F(P) = \sum_{i = 1}^N F(P_i)$ whenever $P$ is built as a gluing of many small nets $P_1, \dots, P_N$. In this paper we completely determine \textit{all} of the additive invariants for composable Petri nets. 
Our results show that additivity forces the invariants to describe simple structural requirements. Thus one can think of a ladder of semantic complexity of isomorphism-invariant mappings on Petri nets where additive invariants occupy lower rungs  compared to  invariants such as mass action kinetics which is a map, defined in~\cite{reaction-networks-pollard-baez}, from Petri nets into differential equations that respects gluing.

\subsection{Contributions}
In this paper, we will classify all $\mathbb{N}$-valued invariants of open Petri nets which are additive with respect to composition and monoidal product in the category of open Petri nets, $\openpetri$. In particular we show that the additive invariants of open Petri nets are completely determined by their values on a particular class of single-transition Petri nets. 

We also give a similar classification for the $\mathbb{N}$-valued invariants for the category of open Petri nets with monic legs, $\mope$, which embeds faithfully into $\openpetri$.  All additive invariants on $\openpetri$ nets are also additive invariants on $\mope$. However, the converse is not true. We show that the additive invariants on $\mope$ are determined by their values on all single-transition Petri nets as well as transitionless Petri nets.

The classification of these additive invariants relies on two decomposition lemmas for open Petri nets. Given the large literature on open Petri nets in applied category theory, these lemmas are of independent significance. 

The paper is structured as follows. In Section~\ref{sec:prelim} we introduce the category of open Petri nets, $\openpetri$ and the category of open Petri nets with monic legs, $\mope$. We then introduce many of the notions that we use in the decomposition lemmas and in classifying invariants. These include particular classes  of transitionless and single transition Petri nets. In Section~\ref{sec:decomposition} we give the decomposition lemmas. In Section~\ref{sec:invariants} we show in Theorem~\ref{thm:all-invariants-are-additive} that in fact all invariants of open Petri nets are additive. Finally, we classify the additive invariants of open Petri nets in Theorem~\ref{thm:open-petri-main-thm} and  of monically open Petri nets in Theorem~\ref{thm:mope-invariant-classification}.

\paragraph{Acknowledgements} We thank John Baez for leading and guiding this project. Furthermore, we would like to thank the American Mathematical Society for its support and the organisers of the AMS' 2022 MRC (Mathematics Research Communities) on Applied Category Theory for setting up this collaboration. Additionally, we thank the referees for a number of insightful remarks that led to the expansion of the discussion on related work, improvement in exposition, and a few simplifications.

\section{Preliminaries} \label{sec:prelim}
\subsection{Open Petri nets}

Open Petri nets are Petri nets equipped with sets of input and output species acting as \textit{interfaces}. Formally, an \defined{open Petri net} is a Petri net \(P\) and a cospan of sets \(X \xrightarrow{i} S \xleftarrow{o} Y\) where $S$ is the species set of the Petri net $P$. The pair \( \Pa = (X \xrightarrow{i} S \xleftarrow{o} Y, \; P)\) is an open Petri net \defined{decorated by} the Petri net $P$. 

Baez and Pollard show~\cite[Theorem 6]{reaction-networks-pollard-baez} that open Petri nets should be thought of as \textit{morphisms} in a \textit{symmetric monoidal category $\openpetri$} of open Petri nets. The objects of this category are finite sets, the morphisms are isomorphism classes of open Petri nets, and composition --- similarly to other cospan categories~\cite{johnson20212} --- is roughly obtained by pushout.
One should think of this as \textit{gluing} Petri nets together along shared interfaces. We invite the reader to consult Baez and Pollard's paper~\cite[Theorem 6]{reaction-networks-pollard-baez} for a thorough account on how to define the category $\openpetri$ by applying Fong's theory of decorated cospans~\cite{fong2016algebra}.

The composition and tensor product of two open Petri nets are applied in many of the proofs throughout Sections~\ref{sec:invariants} and~\ref{sec:decomposition}. For clarity, we give their explicit definition here. Let
\begin{align*}
    \Pa &= \left(X \xrightarrow{i} S \xleftarrow{o} Y, \; \left(s, t\colon T \to \mathbb{N}^S\right) \right), \text{ and} \\
    \Qa &= \left(Y \xrightarrow{i\pr} S\pr \xleftarrow{o\pr} Z, \; \left(s\pr, t\pr\colon T\pr \to \mathbb{N}^{S\pr}\right) \right),
\end{align*}
be two open Petri nets. To compose $\Pa $ and $\Qa $, first we form the pushout:
\[\begin{tikzcd}[column sep={5em,between origins}, row sep={5em,between origins}]
                  &                                       & S+_{Y}S'                            &                                        &                    \\
                  & S \arrow[ru, "\overline{i\pr}", dashed] &                                     & S' \arrow[lu, "\overline{o}"', dashed] &                    \\
X \arrow[ru, "i"] &                                       & Y \arrow[lu, "o"'] \arrow[ru, "i\pr"] &                                        & Z \arrow[lu, "o\pr"']
\arrow["\lrcorner"{anchor=center, pos=.85, rotate=-45, scale = 2}, draw=none, from=3-3, to=1-3]
\end{tikzcd}\]
where the maps $\overline{i\pr}$ and $\overline{o}$ are the canonical morphisms from $S$ and $S'$ to the pushout object $S+_{Y}S'$. 

Then the composite is the open Petri net \[\Pa \cp \Qa \coloneqq \left( X \xrightarrow{\overline{i\pr} \circ i} S +_Y S\pr \xleftarrow{\overline{o} \circ o\pr} Z,\ \left(\bar s, \bar t\colon  T + T\pr \to \mathbb{N}^{S +_Y S\pr}\right) \right)\] where the maps $\bar s, \bar t\colon T + T\pr \to \Nb^{S+_Y S\pr}$ are given by 
\begin{equation}\label{eq:composite-source}
    \bar s(\tau)( \bar \sigma) = \begin{cases} 
        \sum_{\sigma \in S | \overline{i\pr}(\sigma) = \bar \sigma} s(\tau)(\sigma)& \tau \in T\\
        \sum_{\sigma\pr \in S\pr | \overline{o}(\sigma\pr) = \bar \sigma} s\pr(\tau)(\sigma\pr)& \tau \in T\pr\\
    \end{cases},
\end{equation}
and
\begin{equation}\label{eq:composite-target}
    \bar t(\tau)( \bar \sigma) = \begin{cases}  
        \sum_{\sigma \in S | \overline{i\pr}(\sigma) = \bar \sigma} t(\tau)(\sigma)&  \tau \in T\\
        \sum_{\sigma\pr \in S\pr | \overline{o}(\sigma\pr) = \bar \sigma} t\pr(\tau)(\sigma\pr)& \tau \in T\pr\\
    \end{cases}.
\end{equation}

The monoidal product $\oplus$ in $\openpetri$ is defined on objects as their disjoint union. For two open Petri nets  $\Pa = \left(X \xrightarrow{i} S \xleftarrow{o} Y, P = \left(s,t\colon T \to \Nb^S\right)\right)$ and $\Qa = \left(X\pr \xrightarrow{i\pr} S\pr \xleftarrow{o\pr} Y\pr, Q = \left(s\pr,t\pr\colon T\pr \to \Nb^{S\pr}\right)\right)$, their monoidal product is
\begin{equation*}
    \Pa \oplus \Qa\coloneqq\left(X+X\pr\xrightarrow{i+i\pr}S+S\pr\xleftarrow{o+o\pr}Y+Y\pr,\left(s+s\pr,t+t\pr\colon T+T\pr\to \mathbb{N}^{S+S\pr}\right)\right),
\end{equation*}
where $S+S\pr$ and $T+T\pr$ are the disjoint unions of the species and transitions sets of $\Pa$ and $\Qa$, respectively. The map $s+s\pr\colon T+T\pr\to S+S\pr$ sends $\tau\in T$ to $s(\tau)\in S$, and $\tau\pr\in T\pr$ to $s\pr(\tau\pr)\in S\pr$. Similarly, the map $t+t\pr\colon T+T\pr\to S+S\pr$ sends $\tau\in T$ to $t(\tau)\in S$, and $\tau\pr\in T\pr$ to $t\pr(\tau\pr)\in S\pr$.

In addition to invariants of open Petri nets we are also interested in invariants for a subclass of open Petri nets defined as follows. 

\begin{definition}
 A \defined{monically open Petri net} --- or \defined{mope net} for short --- is an open Petri nets whose underlying cospan consists of a pair of monomorphisms. 
\end{definition}

Since the composition of two mope nets is, again, a mope net, these open Petri nets form a category $\mope$ which is a wide subcategory of $\openpetri$.  In $\openpetri$, composing open Petri nets can identify two species in the decoration of one of the factors. For example, an input species and an output species of a transition may be identified so that they appear as a single catalyst in the composite. 
This identification is not possible in $\mope$. Instead, composition in $\mope$ preserves the relation between each transition and its input and output species.

\subsection{Transitionless open Petri nets}

In several constructions throughout the main sections of this paper, we refer to Petri nets with no transitions. For ease of notation, we introduce the following definition.

\begin{definition}\label{def:no-transition-initial-legs}
For a finite set $S$, let $0_S$ denote the unique Petri net with $S$ species and  no transitions. Then the source and target maps are both the unique morphism $!\colon 0 \to \Nb^S$. Explicitly, $0_S \coloneqq (S, 0, !,!)$.
\end{definition}

A \defined{transitionless open Petri net} is an open Petri net whose decoration is transitionless. These are generated from several building blocks as we show in the following lemma.

\begin{lemma}\label{lem:generating-transitionless-petris}
Any transitionless open Petri net is generated via the tensor product and composite from the following (transitionless) open Petri nets:
\begin{align*}
    \mu & := (2 \to 1 \leftarrow 1, 0_1),\\
    \eta & := (0 \to 1 \leftarrow 1, 0_1),\\
    \delta & := ( 1 \to 1 \leftarrow 2, 0_1),\\
    \epsilon & := (1 \to 1 \leftarrow 0, 0_1).\\
\end{align*}
\end{lemma}
\begin{proof}
An open Petri net with no transitions is equivalent to a cospan in $\finset$. By Lemma 3.6 of~\cite{fong2019hypergraph}, cospans in $\finset$ are generated by the cospans underlying the morphisms $\mu$, $\eta$, $\delta$, and $\epsilon$.
\end{proof}

Of particular interest are two types of transitionless open Petri nets.

\begin{definition}\label{def:boundary-nets}%
A \defined{boundary mope net} is a mope net either of the form 
\[
\eta_{S} \coloneqq \left(0 \to S \xleftarrow{1_S} S, 0_S\right) \text{ or } \epsilon_{S} \coloneqq \left(S \xrightarrow{1_S} S \leftarrow 0, 0_S\right),
\]
for a finite set $S$.
\end{definition}

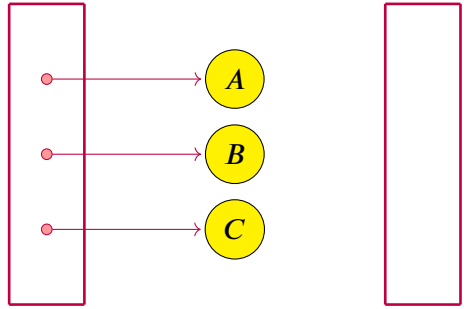
\begin{figure}[h]
    \centering
    \begin{tikzpicture}
\begin{pgfonlayer}{nodelayer}
	\node [style=species] (A) at (0,1) {$A$};
	\node [style=species] (B) at (0,0) {$B$};
	\node [style=species] (C) at (0,-1) {$C$};

	\node [style=none] (Xtr) at (-2, 2) {};
	\node [style=none] (Xbr) at (-2, -2) {};
	\node [style=none] (Xtl) at (-3, 2) {};
    \node [style=none] (Xbl) at (-3, -2) {};
	\node [style=inputdot, fill = red!40, inner sep= 2mm] (1) at (-2.5,1) {};
	\node [style=inputdot, fill = red!40, inner sep= 2mm] (2) at (-2.5,0) {};
	\node [style=inputdot, fill = red!40, inner sep= 2mm] (3) at (-2.5,-1) {};

	\node [style=none] (Ytr) at (3, 2) {};
	\node [style=none] (Ybr) at (3, -2) {};
	\node [style=none] (Ytl) at (2, 2) {};
    \node [style=none] (Ybl) at (2, -2) {};

\end{pgfonlayer}

\begin{pgfonlayer}{edgelayer}
	\draw [style=simple] (Xtl.center) to (Xtr.center);
	\draw [style=simple] (Xtr.center) to (Xbr.center);
	\draw [style=simple] (Xbr.center) to (Xbl.center);
	\draw [style=simple] (Xbl.center) to (Xtl.center);

	\draw [style=simple] (Ytl.center) to (Ytr.center);
	\draw [style=simple] (Ytr.center) to (Ybr.center);
	\draw [style=simple] (Ybr.center) to (Ybl.center);
	\draw [style=simple] (Ybl.center) to (Ytl.center);

	\draw [style=inputarrow] (1) to (A);
	\draw [style=inputarrow] (2) to (B);
	\draw [style=inputarrow] (3) to (C);
\end{pgfonlayer}
\end{tikzpicture}
    \caption{The boundary mope net $\epsilon_3$.}
    \label{fig:boundary_net}
\end{figure}

\begin{remark} \label{rem01}
Note that the boundary mope nets $\eta_1$ and $\epsilon_1$ are in fact the generators $\eta$ and $\epsilon$ defined in Lemma~\ref{lem:generating-transitionless-petris}.
Additionally, there are several useful relations between these boundary mope nets. 
\begin{align*}
    \delta \cp \mu &= \id_1,\\
    \eta_2 \cp \mu & = \eta_1\\
    \delta \cp \epsilon_2 &= \epsilon_1\\
\end{align*}
Furthermore for a finite set $S$, we have
\[
    \underbrace{\eta_1 \oplus \cdots \oplus \eta_1}_{|S| \text{ times}} = \eta_S
\] and 
\[
    \underbrace{\epsilon_1 \oplus \cdots \oplus \epsilon_1}_{|S| \text{ times}} = \epsilon_S.
\]

\end{remark}


\subsection{Atomic Petri nets}

Next we introduce classes of open Petri nets which will form the building blocks for our decomposition lemmas in Section~\ref{sec:decomposition}.

\begin{definition}
    An \defined{atomic} Petri net is a Petri net with a single transition  such that each species is connected to the transition as an input and/or as an output.
\end{definition}

\begin{definition}
    For integers $m,n$ define $P_{m,n}$ to be the atomic Petri net whose transition has $m$ distinct input species and $n$ distinct output species.

    Explicitly, $P_{m,n}$ has a single transition $\tau$ and $m + n$ distinct species.

    \[
        S = \{1,..., m + n\}
    \] with source and target maps 
    \[
        s(\tau)(i) = \begin{cases}
        1 & i = 1,..., m\\
        0 & i = m +1 ,..., m+n\\
        \end{cases}, \quad 
        t(\tau)(i) = \begin{cases}
        0 & i = 1,..., m\\
        1 & i = m +1 ,..., m+n
        \end{cases}.
    \]
\end{definition}

\begin{example}
    Figure~\ref{fig:atomic-examples} gives four examples of Petri nets with a single transition. The Petri nets in (a), (b), and (c) are all atomic. The Petri nets in Figure~\ref{fig:atomic-examples} (a) and (b) are not of the form $P_{m,n}$ for any $m,n$. In (a), this is because there is a species that is both an input species and an output species for the transition. We call this type of species a \defined{catalyst} of the transition. On the other hand, (b) is not of the form $P_{m,n}$ because there is a species connected to the transition by two output arcs.  Finally, (c) depicts the Petri net $P_{2,2}$. It has two distinct input species, each of which is connected to the transition by a single input arc. Likewise, it has two distinct output species, each of which is connected to the transition by a single output arc. 
    \par{}
    Finally the Petri net in (d) it not atomic because it contains a species that does not participate in the transition.
\end{example}

\begin{figure}
    \[
    \substack{\begin{tikzpicture}
	\begin{pgfonlayer}{nodelayer}
		\node [style=species, inner sep= 2mm] (A) at (-1.5,0){};
		\node [style=species, inner sep= 2mm] (B) at (0,2){};
		\node [style=species, inner sep= 2mm] (C) at (1.5,0){};
		\node [style=transition, inner sep= 2mm, fill=green!40] (a) at (0,0){};
	\end{pgfonlayer}

	\begin{pgfonlayer}{edgelayer}
	\end{pgfonlayer}
		\draw [style=inarrow ] (A) to (a);
		\draw [style=inarrow , bend left=35, looseness=1.00] (B) to (a);
		\draw [style=inarrow , bend left=35,looseness=1.00] (a) to (B);
		\draw [style=inarrow ] (a) to (C);
\end{tikzpicture}\\\\{\text{\Large(a)}}}
    \qquad\qquad
    \substack{\begin{tikzpicture}
	\begin{pgfonlayer}{nodelayer}
		\node [style=species , inner sep= 2mm] (A) at (-1.5,1){};
		\node [style=species , inner sep= 2mm] (B) at (-1.5,-1){};
		\node [style=species , inner sep= 2mm] (C) at (1.5,0){};
		\node [style=transition , inner sep= 2mm, fill=green!40] (a) at (0,0){};
	\end{pgfonlayer}

	\begin{pgfonlayer}{edgelayer}
	\end{pgfonlayer}
		\draw [style=inarrow] (A) to (a);
		\draw [style=inarrow] (B) to (a);
		\draw [style=inarrow,bend left=35,looseness=1.00] (a) to (C);
		\draw [style=inarrow,bend right=35,looseness=1.00] (a) to (C);
\end{tikzpicture}\\\\{\text{\Large(b)}}}
    \]
	\vspace{0.2in}
    \[
    \substack{\begin{tikzpicture}
	\begin{pgfonlayer}{nodelayer}
		\node [style=species , inner sep= 2mm] (A) at (-1.5,1){};
		\node [style=species , inner sep= 2mm] (B) at (-1.5,-1){};
		\node [style=species , inner sep= 2mm] (C) at (1.5,1){};
		\node [style=species , inner sep= 2mm] (D) at (1.5,-1){};
		\node [style=transition , inner sep= 2mm, fill=green!40] (a) at (0,0){};
	\end{pgfonlayer}

	\begin{pgfonlayer}{edgelayer}
	\end{pgfonlayer}
		\draw [style=inarrow] (A) to (a);
		\draw [style=inarrow] (B) to (a);
		\draw [style=inarrow] (a) to (C);
		\draw [style=inarrow] (a) to (D);
\end{tikzpicture}\\\\{\text{\Large(c)}}}
    \qquad\qquad
    \substack{\begin{tikzpicture}
	\begin{pgfonlayer}{nodelayer}
		\node [style=species , inner sep= 2mm] (A) at (-1.5,0){};
		\node [style=species , inner sep= 2mm] (B) at (-1.5,-2){};
		\node [style=species , inner sep= 2mm] (C) at (1.5,-1){};
		\node [style=species, inner sep= 2mm] (D) at (0,0){};
		\node [style=transition , inner sep= 2mm, fill=green!40] (a) at (0,-1){};
	\end{pgfonlayer}

	\begin{pgfonlayer}{edgelayer}
	\end{pgfonlayer}
		\draw [style=inarrow] (A) to (a);
		\draw [style=inarrow] (B) to (a);
		\draw [style=inarrow] (a) to (C);
\end{tikzpicture}\\\\{\text{\Large(d)}}}
    \]
 \caption{(a) An atomic Petri net containing a catalyst. (b) An atomic Petri net with a single transition that has two output arcs connected to the same species. (c) The atomic Petri net $P_{2,2}$. (d) A non-atomic Petri net with a single transition that has a species that does not participate in the transition.}
    \label{fig:atomic-examples}
\end{figure}
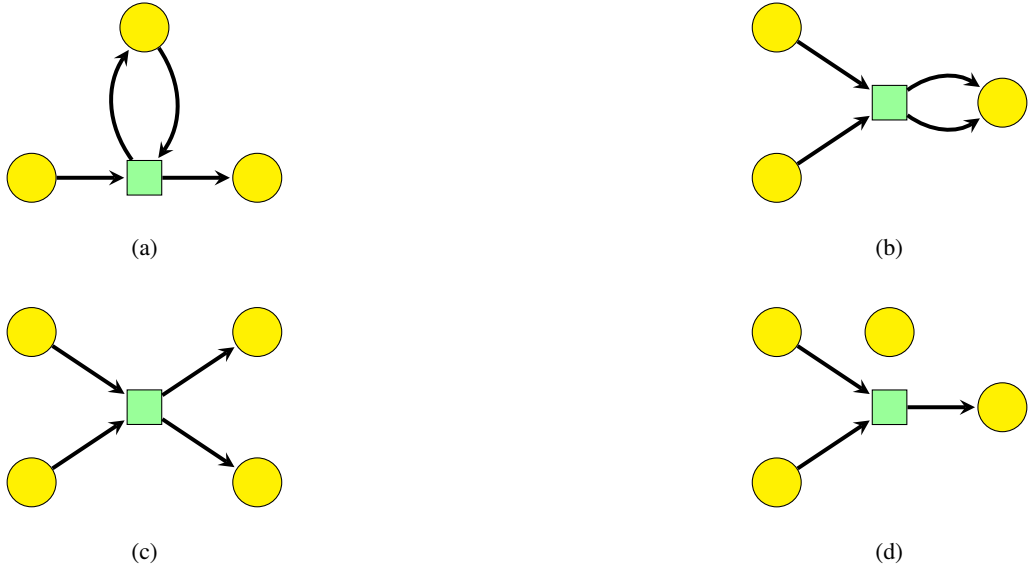

Each atomic Petri net has a type signature called a \defined{body type} which we define below. 

\begin{definition}\label{def:body-type}
Let $P = (s,t: T \to \Nb^S)$ by any Petri net. For a transition $\tau \in T$, let $S_\tau\subseteq S$ be the subset of species which participate in the transition $\tau$. So $S_\tau = \{\sigma \in S | s(\tau)(\sigma) \neq 0  \text{ or } t(\tau)(\sigma) \neq 0\}$. Then the transition $\tau \in T$ is represented by the
multiset
\[
    \{ [ (s(\tau)(\sigma), t(\tau)(\sigma)) ] | \sigma \in S_\tau \}
\]
of pairs of natural numbers $\Nn \times \Nn \setminus \{(0,0)\}$.
We refer to this multiset as the \defined{body type} of a transition $\tau$.
\end{definition}

There is a 1-1 correspondence between body types and isomorphism classes of atomic Petri nets. Indeed, consider the body type 
$\bt = [(a_1,b_1),..., (a_n, b_n)]$. Then, 
let $P_\bt$ be an atomic Petri net with a transition $\tau$ and $n$ distinct species $S = \{1,..., n\}$ with source and target maps
$
    s(\tau)(i) = a_i, \  t(\tau)(i) =  b_i.
$
The Petri net $P_\bt$  is a canonical representative of the isomorphism class of atomic Petri nets that correspond to body type $\bt$. 

\begin{remark}
    Note that $P_{m,n} = P_\bt$ for body type 
    \[
        \bt = [\underbrace{(1, 0), \cdots , (1,0)}_{m\text{-times}}, \underbrace{(0, 1) \cdots , (0, 1)}_{n \text{-times}}].
    \]
\end{remark}

\begin{definition}
    For natural numbers $m,n$, let $\Pa_{m,n}$ be the open Petri net decorated by $P_{m,n}$ and whose underlying cospan is the identity.

    Likewise for body type $\bt$, let $\Pa_\bt$ be the open Petri net decorated by $P_\bt$ and whose underlying cospan is the identity. We call an open  Petri net of this form a \defined{body net}.
\end{definition}

Figure~\ref{fig:body_net} gives an example of the body net $\Pa_{1,1}$.

\begin{remark}\label{rmk:atomic}
    If $\Pa$ is any open Petri net whose decoration has a single transition and whose underlying cospan is the identity, then $\Pa$ is isomorphic to the monoidal product of a body net $\Pa_\bt$ and an identity open Petri net.
\end{remark}

\begin{figure}[h]
    \centering
    \begin{tikzpicture}
\begin{pgfonlayer}{nodelayer}
	\node [style=species, inner sep= 2mm] (A) at (0,1) {};
	\node [style=species, inner sep= 2mm] (B) at (0,-1) {};
	\node [style=transition, inner sep= 2mm, fill=green!40] (t) at (0,0) {};

 \node [style=empty ] (X) at (-2.5, 2.3) {${\color{gray}X}$};
	\node [style=none] (Xtr) at (-2, 2) {};
	\node [style=none] (Xbr) at (-2, -2) {};
	\node [style=none] (Xtl) at (-3, 2) {};
    \node [style=none] (Xbl) at (-3, -2) {};
	\node [style=inputdot, fill = red!40, inner sep= 2mm] (1) at (-2.5,1) {};
	\node [style=inputdot, fill = red!40, inner sep= 2mm] (2) at (-2.5,-1) {};

 \node [style=empty] (Y) at (2.5, 2.3) {${\color{gray} X}$};
	\node [style=none] (Ytr) at (3, 2) {};
	\node [style=none] (Ybr) at (3, -2) {};
	\node [style=none] (Ytl) at (2, 2) {};
    \node [style=none] (Ybl) at (2, -2) {};
	\node [style=inputdot, fill = red!40, inner sep= 2mm] (3) at (2.5,1) {};
	\node [style=inputdot, fill = red!40, inner sep= 2mm] (4) at (2.5,-1) {};

\end{pgfonlayer}

\begin{pgfonlayer}{edgelayer}
	\draw [style=simple] (Xtl.center) to (Xtr.center);
	\draw [style=simple] (Xtr.center) to (Xbr.center);
	\draw [style=simple] (Xbr.center) to (Xbl.center);
	\draw [style=simple] (Xbl.center) to (Xtl.center);

	\draw [style=simple] (Ytl.center) to (Ytr.center);
	\draw [style=simple] (Ytr.center) to (Ybr.center);
	\draw [style=simple] (Ybr.center) to (Ybl.center);
	\draw [style=simple] (Ybl.center) to (Ytl.center);

	\draw [style=inputarrow] (1) to (A);
	\draw [style=inputarrow] (2) to (B);

	\draw [style=inputarrow] (3) to (A);
	\draw [style=inputarrow] (4) to (B);

	\draw [style=inarrow] (A) to (t);
	\draw [style=inarrow] (t) to (B);
\end{pgfonlayer}
\end{tikzpicture}
    \caption{The open atomic Petri net $\Pa_{1,1}$.}
    \label{fig:body_net}
\end{figure}
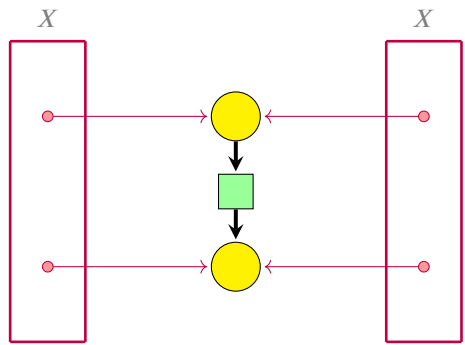

\section{Decomposition Lemmas}\label{sec:decomposition}
Petri nets in the wild are often quite complicated, with hundreds of transitions.
In this section we first prove a decomposition theorem for open Petri nets which decomposes an open Petri
net into  transitionless and single-transition open Petri nets. One of the main advantages of this  decomposition theorem is that it applies to mope nets as well as generic open Petri nets.
Then we will prove Lemma~\ref{lemma:decomp} which gives a canonical decomposition of open Petri nets into transitionless open Petri nets and open Petri nets decorated with Petri nets of the form $P_{m,n}$. This factorization involves both composition and the monoidal product.

We start first by showing that any open Petri net is canonically the composite of   transitionless open Petri nets and an open Petri net whose underlying cospan is the identity.

\begin{lemma}\label{lem:extract-net}
Let $\mathcal{P}$ be an open Petri net. Then $\mathcal{P}$ can be decomposed as 
\[ 
    \mathcal{P} = \mathcal{Q} \cp \mathcal{R} \cp   \mathcal{Q'},
\]
where $\mathcal{Q}$ and $\mathcal{Q'}$ are decorated by transitionless Petri nets and the legs of $\mathcal{R}$ are  identities. 
If $\mathcal P$ is monic, then so are $\mathcal Q$ and $\mathcal Q'$.
\end{lemma}
\begin{proof}
Let $\mathcal{P}$ be the  Petri net defined by
\[
    \mathcal{P} = \left(X \xrightarrow{i} S \xleftarrow{o} Y, P=\left(S,T, s,t \colon T \to \mathbb{N}^S\right)\right).
\]
Consider the  Petri nets 
\[
    \mathcal{Q} =\left(X \xrightarrow{i} S \xleftarrow{1_S} S, 0_S\right), 
\]  
\[
    \mathcal{R} = \left(S \xrightarrow{1_S} S \xleftarrow{1_S} S, P\right),
\] and 
\[
    \mathcal{Q}' = \left(S \xrightarrow{1_S} S \xleftarrow{o} Y, 0_S\right).
\]
Note that if $\mathcal P$ is monic, then so are $\mathcal Q$ and $\mathcal Q'$.

First, we show that $\Qa \cp \Ra = (X \xrightarrow{i} S \xleftarrow{1_S} S, P)$. Consider the composition of cospans:

\[\begin{tikzcd}
	&& S \\
	& S && S \\
	X && S && S
	\arrow["i", from=3-1, to=2-2]
	\arrow["{1_S}"', from=3-3, to=2-2]
	\arrow["{1_S}", from=3-3, to=2-4]
	\arrow["{1_S}"', from=3-5, to=2-4]
	\arrow["{1_S}", dashed, from=2-2, to=1-3]
	\arrow["{1_S}"', dashed, from=2-4, to=1-3]
	\arrow["\lrcorner"{anchor=center, pos=0.125, rotate=-45, scale = 2}, draw=none, from=1-3, to=3-3]
    \end{tikzcd}\]

That the decoration is $P$ follows straightforwardly from the explicit definition of the composition of open Petri nets given in Section~\ref{sec:prelim}.

Similarly $(\mathcal{Q} \cp \mathcal{R}) \cp  \mathcal{Q'}$ consists of the cospan $X \xrightarrow{i} S \xleftarrow{o} Y$ and the decoration $P$.

\end{proof}

We are now ready to prove 
our first decomposition Lemma of open Petri nets into atomic factors.

\begin{lemma}\label{lemma:decomp-atomic}
If $\Pa$ is an open Petri net with $N$ transitions, then 
\[
    \Pa = \mathcal{Q}\cp \mathcal{G}_{1} \cp \mathcal{G}_{2} \cp \cdots \cp \mathcal{G}_{N} \cp \mathcal{Q}\pr,
\]
where $\mathcal{Q}$ and $\mathcal{Q\pr}$ are transitionless open Petri nets and each $\mathcal{G}_{i}$ satisfies: 
\begin{itemize}
    \item Its underlying cospan is the identity.
    \item Its decoration has a single transition.
\end{itemize}
\end{lemma}

\begin{proof}
By Lemma~\ref{lem:extract-net} is suffices to prove the theorem for open Petri nets whose underlying cospan is the identity. We do this by induction on $N$.  Suppose $\Pa$ is of the form
$$\mathcal{P} = (S \xrightarrow{1_S} S \xleftarrow{1_S} S, P=(S,\ T,\ s,\ t\colon\ T \to \mathbb{N}^S)),$$
where $T$ has elements $\tau_1, \tau_2, \ldots, \tau_N$.

If $N=0$ or $N=1$ then the result is trivial. Now suppose that $N \geq 2$ and that the result holds for all open Petri nets with fewer than $N$ transitions.

Define $\tilde T = \{\tau_1,..., \tau_{N-1}\}$ and the Petri nets 
\[ \tilde P = (S, \tilde T, s\rest_{\tilde T}, t\rest_{\tilde T}), \quad G_N = (S, \{\tau_N\}, s\rest_{\{\tau_N\}}, t\rest_{\{\tau_N\}}).
\] Let $\Ga_N$ be the open Petri net $(S \xrightarrow{1_S} S \xleftarrow{1_S} S, G_N)$. By construction $\Ga_N$ satisfies the requisite criteria. Then $\Pa = (S\xrightarrow{1_S} S \xleftarrow{1_S} S, \tilde P) \cp \Ga_N$. This result is a straightforward application of the definition of composing open Petri nets. 

By the induction hypothesis there exist  open Petri nets $\Ga_1, \cdots , \Ga_{N-1}$ satisfying the criteria such that 
\[
    (S\xrightarrow{1_S} S \xleftarrow{1_S} S, \tilde P) = \Ga_1 \cp \cdots \cp \Ga_{N - 1}.
\] Thus $\Pa = \Ga_1 \cp \cdots \cp \Ga_{N - 1} \cp \Ga_N$.
\end{proof}

\begin{remark}\label{rmk:decomp-order}
    The order of the Petri nets $\Ga_i$ depended on the choice of ordering of the transitions. Therefore, this decomposition is unique up to isomorphism and permutations of $\Ga_i$. Since composition involves pushouts, everything is unique up to isomorphism.
\end{remark}

\begin{remark}
    By Lemma~\ref{lem:extract-net}, if $\Pa$ is a mope net then its factors according to the decomposition in Lemma~\ref{lemma:decomp-atomic} are mope nets as well.
\end{remark}

\begin{remark}
    By Remark~\ref{rmk:atomic}, each $\Ga_i$ is the monoidal product of a body net and an identity open Petri net. 
\end{remark}

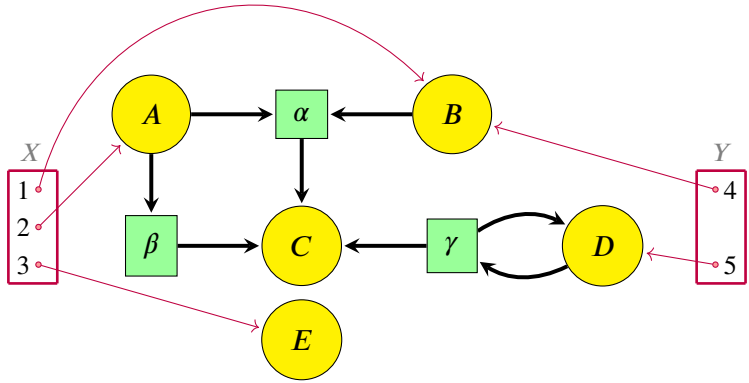
\begin{figure}[h]
    \centering
    \[
    \begin{tikzpicture}
	\begin{pgfonlayer}{nodelayer}
		\node [style=species, inner sep= 2mm] (A) at (-3.5, 1.5) {$A$};
		\node [style=species, inner sep= 2mm] (B) at (0.5, 1.5) {$B$};
		\node [style=species, inner sep= 2mm] (C) at (-1.5, -0.25) {$C$};
		\node [style=species, inner sep= 2mm] (D) at (2.5, -0.25) {$D$};
		\node [style=species, inner sep= 2mm] (E) at (-1.5,-1.5) {$E$};
        
        \node [style=transition, inner sep= 2mm, fill=green!40] (a) at (-1.5, 1.5) {$\alpha$};
        \node [style=transition, inner sep= 2mm, fill=green!40] (b) at (-3.5, -0.25) {$\beta$};
        \node [style=transition, inner sep= 2mm, fill=green!40] (c) at (0.5, -0.25) {$\gamma$}; 
		
		\node [style=empty] (X) at (-5.1, 1) {${\color{gray}X}$};
		\node [style=none] (Xtr) at (-4.75, 0.75) {};
		\node [style=none] (Xbr) at (-4.75, -0.75) {};
		\node [style=none] (Xtl) at (-5.4, 0.75) {};
        \node [style=none] (Xbl) at (-5.4, -0.75) {};
	
		\node [style=inputdot,fill = red!40, inner sep= 1mm] (1) at (-5, 0.5) {};
		\node [style=empty] at (-5.2, 0.5) {$1$};
		\node [style=inputdot, fill = red!40, inner sep= 1mm] (2) at (-5, 0) {};
		\node [style=empty] at (-5.2, 0) {$2$};
		\node [style=inputdot, fill = red!40, inner sep= 1mm] (3) at (-5, -0.5) {};
		\node [style=empty] at (-5.2, -0.5) {$3$};

		\node [style=empty] (Y) at (4.1, 1) {${\color{gray}Y}$};
		\node [style=none] (Ytr) at (4.4, 0.75) {};
		\node [style=none] (Ytl) at (3.75, 0.75) {};
		\node [style=none] (Ybr) at (4.4, -0.75) {};
		\node [style=none] (Ybl) at (3.75, -0.75) {};

		\node [style=inputdot, fill = red!40, inner sep= 1mm] (4) at (4, 0.5) {};
		\node [style=empty] at (4.2, 0.5) {$4$};
		\node [style=inputdot, fill = red!40, inner sep= 1mm] (5) at (4, -0.5) {};
		\node [style=empty] at (4.2, -0.5) {$5$};		
		
		
	\end{pgfonlayer}
	\begin{pgfonlayer}{edgelayer}
		\draw [style=inarrow] (A) to (a);
        \draw [style=inarrow] (A) to (b);
		\draw [style=inarrow] (B) to (a);
        \draw [style=inarrow,bend left=30, looseness=1.00] (D) to (c);
		\draw [style=inarrow] (a) to (C);
        \draw [style=inarrow] (b) to (C);
        \draw [style=inarrow] (c) to (C);
        \draw [style=inarrow,bend left=30, looseness=1.00] (c) to (D);

		\draw [style=inputarrow,bend left=60, looseness=1.25] (1) to (B);
		\draw [style=inputarrow] (2) to (A);
		\draw [style=inputarrow] (3) to (E);
		\draw [style=inputarrow] (4) to (B);
        \draw [style=inputarrow] (5) to (D);
		\draw [style=simple] (Xtl.center) to (Xtr.center);
		\draw [style=simple] (Xtr.center) to (Xbr.center);
		\draw [style=simple] (Xbr.center) to (Xbl.center);
		\draw [style=simple] (Xbl.center) to (Xtl.center);
		\draw [style=simple] (Ytl.center) to (Ytr.center);
		\draw [style=simple] (Ytr.center) to (Ybr.center);
		\draw [style=simple] (Ybr.center) to (Ybl.center);
		\draw [style=simple] (Ybl.center) to (Ytl.center);
	\end{pgfonlayer}
\end{tikzpicture}
\]
    \caption{A mope net with three transitions. 
    }
    \label{fig:example-mope}
\end{figure} 
\begin{example}
Consider the mope net $\mathcal{P} = \left(X \xrightarrow{i} S \xleftarrow{o} Y, P= \left(S,T,s,t\colon T \to \mathbb{N}^S\right)\right)$, where $S=\{ A,B,C,D,E\}$ and $T = \{ \alpha,\beta,\gamma \}$ shown in Figure~\ref{fig:example-mope}. Let 
\[
\begin{split}
    \mathcal{Q} & = \left(X \xrightarrow{i} S \xleftarrow{1_S} S, 0_S\right), \text{ and}\\
    \mathcal{Q'} & = \left(S \xrightarrow{1_S} S \xleftarrow{o} Y, 0_S\right).\\
    \end{split}
\]
For $i=1,2,3$, let 
\[
    \mathcal{G}_i= \left(S \xrightarrow{1_S} S \xleftarrow{1_S} S, P_i = \left(S, T_i,s\rest_{T_i},t\rest_{T_i}\colon T_i \to \mathbb{N}^S\right)\right),
\] where $T_{1} = \{\alpha\}, T_{2} = \{\beta\}$, and $T_{3}=\{\gamma\}$. Using Lemma~\ref{lemma:decomp-atomic},  $\mathcal{P}$ can be decomposed as follows 
\[
    \mathcal{P} = \mathcal{Q} \cp \mathcal{G}_1 \cp \mathcal{G}_2 \cp \mathcal{G}_3 \cp \mathcal{Q'}.
\]
See Figure~\ref{fig:example-mdecomposition} for a depiction of this decomposition.

As an example of Remark~\ref{rmk:decomp-order}, this decomposition does not depend on the order of $\Ga_1, \Ga_2, \Ga_3$. Composing them in a permuted order yields the same composite mope net.

\end{example}

\begin{figure}[h]
    \centering
    \resizebox{1\textwidth}{!}{%
    \begin{tikzpicture}[scale=0.6]
\begin{pgfonlayer}{nodelayer}
	\node [style=empty] (X1) at (-15.1, 1) {};
	\node [style=none] (X1tr) at (-15.75, 4.25) {};
	\node [style=none] (X1br) at (-15.75, -4.25) {};
	\node [style=none] (X1tl) at (-16.25, 4.25) {};
    \node [style=none] (X1bl) at (-16.25, -4.25) {};

	\node [style=species] (A1) at (-14,4) {$A$};
	\node [style=species] (B1) at (-14,2) {$B$};
	\node [style=species] (C1) at (-14,0) {$C$};
	\node [style=species] (D1) at (-14,-2) {$D$};
	\node [style=species] (E1) at (-14,-4) {$E$};
	
	\node [style=inputdot, fill = red!40, inner sep= 1mm] (1) at (-16, 4) {};
	\node [style=inputdot, fill = red!40, inner sep= 1mm] (2) at (-16, 2) {};
	\node [style=inputdot, fill = red!40, inner sep= 1mm] (3) at (-16, -2) {};

	\node [style=empty] (X2) at (-11.1, 1) {};
	\node [style=none] (X2tr) at (-11.75, 4.25) {};
	\node [style=none] (X2br) at (-11.75, -4.25) {};
	\node [style=none] (X2tl) at (-12.25, 4.25) {};
    \node [style=none] (X2bl) at (-12.25, -4.25) {};

	\node [style=inputdot, fill = red!40, inner sep= 1mm] (5) at (-12, 4) {};
	\node [style=inputdot, fill = red!40, inner sep= 1mm] (6) at (-12, 2) {};
	\node [style=inputdot, fill = red!40, inner sep= 1mm] (7) at (-12, 0) {};
	\node [style=inputdot, fill = red!40, inner sep= 1mm] (8) at (-12, -2) {};
	\node [style=inputdot, fill = red!40, inner sep= 1mm] (9) at (-12, -4) {};

	\node [style=species] (A2) at (-10,4) {$A$};
	\node [style=species] (B2) at (-6,4) {$B$};
	\node [style=species] (C2) at (-8,0) {$C$};
	\node [style=species] (D2) at (-8,-2) {$D$};
	\node [style=species] (E2) at (-8,-4) {$E$};

	\node [style=transition, inner sep= 2mm, fill=green!40] (a) at (-8,2) {$\alpha$};

	\node [style=empty] (X3) at (-4.1, 1) {};
	\node [style=none] (X3tr) at (-3.75, 4.25) {};
	\node [style=none] (X3br) at (-3.75, -4.25) {};
	\node [style=none] (X3tl) at (-4.25, 4.25) {};
    \node [style=none] (X3bl) at (-4.25, -4.25) {};

	\node [style=inputdot, fill = red!40, inner sep= 1mm] (10) at (-4, 4) {};
	\node [style=inputdot, fill = red!40, inner sep= 1mm] (11) at (-4, 2) {};
	\node [style=inputdot, fill = red!40, inner sep= 1mm] (12) at (-4, 0) {};
	\node [style=inputdot, fill = red!40, inner sep= 1mm] (13) at (-4, -2) {};
	\node [style=inputdot, fill = red!40, inner sep= 1mm] (14) at (-4, -4) {};	

	\node [style=species] (A3) at (-2,4) {$A$};
	\node [style=species] (B3) at (2,4) {$B$};
	\node [style=species] (C3) at (0,0) {$C$};
	\node [style=species] (D3) at (0,-2) {$D$};
	\node [style=species] (E3) at (0,-4) {$E$};

	\node [style=transition, inner sep= 2mm, fill=green!40] (b) at (0,2) {$\beta$};

	\node [style=empty] (X4) at (3.9, 1) {};
	\node [style=none] (X4tr) at (4.25, 4.25) {};
	\node [style=none] (X4br) at (4.25, -4.25) {};
	\node [style=none] (X4tl) at (3.75, 4.25) {};
    \node [style=none] (X4bl) at (3.75, -4.25) {};

	\node [style=inputdot, fill = red!40, inner sep= 1mm] (15) at (4, 4) {};
	\node [style=inputdot, fill = red!40, inner sep= 1mm] (16) at (4, 2) {};
	\node [style=inputdot, fill = red!40, inner sep= 1mm] (17) at (4, 0) {};
	\node [style=inputdot, fill = red!40, inner sep= 1mm] (18) at (4, -2) {};
	\node [style=inputdot, fill = red!40, inner sep= 1mm] (19) at (4, -4) {};	

	\node [style=species] (A4) at (6,4) {$A$};
	\node [style=species] (B4) at (10,4) {$B$};
	\node [style=species] (C4) at (8,2) {$C$};
	\node [style=species] (D4) at (8,-2) {$D$};
	\node [style=species] (E4) at (8,-4) {$E$};

	\node [style=transition, inner sep= 2mm, fill=green!40] (c) at (8,0) {$\gamma$};

	\node [style=empty] (X5) at (12.9, 1) {};
	\node [style=none] (X5tr) at (12.25, 4.25) {};
	\node [style=none] (X5br) at (12.25, -4.25) {};
	\node [style=none] (X5tl) at (11.75, 4.25) {};
    \node [style=none] (X5bl) at (11.75, -4.25) {};

	\node [style=inputdot, fill = red!40, inner sep= 1mm] (20) at (12, 4) {};
	\node [style=inputdot, fill = red!40, inner sep= 1mm] (21) at (12, 2) {};
	\node [style=inputdot, fill = red!40, inner sep= 1mm] (22) at (12, 0) {};
	\node [style=inputdot, fill = red!40, inner sep= 1mm] (23) at (12, -2) {};
	\node [style=inputdot, fill = red!40, inner sep= 1mm] (24) at (12, -4) {};

	\node [style=species] (A5) at (14,4) {$A$};
	\node [style=species] (B5) at (14,2) {$B$};
	\node [style=species] (C5) at (14,0) {$C$};
	\node [style=species] (D5) at (14,-2) {$D$};
	\node [style=species] (E5) at (14,-4) {$E$};

	\node [style=empty] (X6) at (12.9, 1) {};
	\node [style=none] (X6tr) at (16.25, 4.25) {};
	\node [style=none] (X6br) at (16.25, -4.25) {};
	\node [style=none] (X6tl) at (15.75, 4.25) {};
    \node [style=none] (X6bl) at (15.75, -4.25) {};
	
	\node [style=inputdot, fill = red!40, inner sep= 1mm] (25) at (16, 2) {};
	\node [style=inputdot, fill = red!40, inner sep= 1mm] (26) at (16, -2) {};

\end{pgfonlayer}

\begin{pgfonlayer}{edgelayer}
	\draw [style=simple] (X1tl.center) to (X1tr.center);
	\draw [style=simple] (X1tr.center) to (X1br.center);
	\draw [style=simple] (X1br.center) to (X1bl.center);
	\draw [style=simple] (X1bl.center) to (X1tl.center);

	\draw [style=simple] (X2tl.center) to (X2tr.center);
	\draw [style=simple] (X2tr.center) to (X2br.center);
	\draw [style=simple] (X2br.center) to (X2bl.center);
	\draw [style=simple] (X2bl.center) to (X2tl.center);

	\draw [style=simple] (X3tl.center) to (X3tr.center);
	\draw [style=simple] (X3tr.center) to (X3br.center);
	\draw [style=simple] (X3br.center) to (X3bl.center);
	\draw [style=simple] (X3bl.center) to (X3tl.center);

	\draw [style=simple] (X4tl.center) to (X4tr.center);
	\draw [style=simple] (X4tr.center) to (X4br.center);
	\draw [style=simple] (X4br.center) to (X4bl.center);
	\draw [style=simple] (X4bl.center) to (X4tl.center);

	\draw [style=simple] (X5tl.center) to (X5tr.center);
	\draw [style=simple] (X5tr.center) to (X5br.center);
	\draw [style=simple] (X5br.center) to (X5bl.center);
	\draw [style=simple] (X5bl.center) to (X5tl.center);

	\draw [style=simple] (X6tl.center) to (X6tr.center);
	\draw [style=simple] (X6tr.center) to (X6br.center);
	\draw [style=simple] (X6br.center) to (X6bl.center);
	\draw [style=simple] (X6bl.center) to (X6tl.center);

	\draw [style=inputarrow] (1) to (B1);
	\draw [style=inputarrow] (2) to (A1);
	\draw [style=inputarrow] (3) to (E1);

	\draw [style=inputarrow] (5) to (A1);
	\draw [style=inputarrow] (6) to (B1);
	\draw [style=inputarrow] (7) to (C1);
	\draw [style=inputarrow] (8) to (D1);
	\draw [style=inputarrow] (9) to (E1);

	\draw [style=inputarrow] (5) to (A2);
	\draw [style=inputarrow] (6) to (B2);
	\draw [style=inputarrow] (7) to (C2);
	\draw [style=inputarrow] (8) to (D2);
	\draw [style=inputarrow] (9) to (E2);

	\draw [style=inputarrow, bend right=35, looseness=1.00] (10) to (A2);
	\draw [style=inputarrow] (11) to (B2);
	\draw [style=inputarrow] (12) to (C2);
	\draw [style=inputarrow] (13) to (D2);
	\draw [style=inputarrow] (14) to (E2);

	\draw [style=inputarrow] (10) to (A3);
	\draw [style=inputarrow] (11) to (B3);
	\draw [style=inputarrow] (12) to (C3);
	\draw [style=inputarrow] (13) to (D3);
	\draw [style=inputarrow] (14) to (E3);

	\draw [style=inputarrow, bend right=35, looseness=1.00] (15) to (A3);
	\draw [style=inputarrow] (16) to (B3);
	\draw [style=inputarrow] (17) to (C3);
	\draw [style=inputarrow] (18) to (D3);
	\draw [style=inputarrow] (19) to (E3);

	\draw [style=inputarrow] (15) to (A4);
	\draw [style=inputarrow] (16) to (B4);
	\draw [style=inputarrow] (17) to (C4);
	\draw [style=inputarrow] (18) to (D4);
	\draw [style=inputarrow] (19) to (E4);

	\draw [style=inputarrow, bend right=35, looseness=1.00] (20) to (A4);
	\draw [style=inputarrow] (21) to (B4);
	\draw [style=inputarrow] (22) to (C4);
	\draw [style=inputarrow] (23) to (D4);
	\draw [style=inputarrow] (24) to (E4);

	\draw [style=inputarrow] (20) to (A5);
	\draw [style=inputarrow] (21) to (B5);
	\draw [style=inputarrow] (22) to (C5);
	\draw [style=inputarrow] (23) to (D5);
	\draw [style=inputarrow] (24) to (E5);

	\draw [style=inputarrow] (25) to (B5);
	\draw [style=inputarrow] (26) to (D5);

	\draw [style=inarrow] (A2) to (a);
	\draw [style=inarrow] (B2) to (a);
    \draw [style=inarrow] (a) to (C2);

	\draw [style=inarrow] (A3) to (b);
	\draw [style=inarrow] (b) to (C3);
	\draw [style=inarrow, bend right=35, looseness=1.00] (c) to (D4);
	\draw [style=inarrow, bend right=35, looseness=1.00] (D4) to (c);
	\draw [style=inarrow] (c) to (C4);
\end{pgfonlayer}
\end{tikzpicture}
}
    \caption{The decomposition of the mope net depicted in Figure~\ref{fig:example-mope} into transitionless mope nets $\Qa$ and $\Qa'$ (farthest left and farthest right) and single-transition  mope nets $\Ga_1$, $\Ga_2$, and $\Ga_3$ (middle mope nets). 
    }
    \label{fig:example-mdecomposition}
\end{figure}
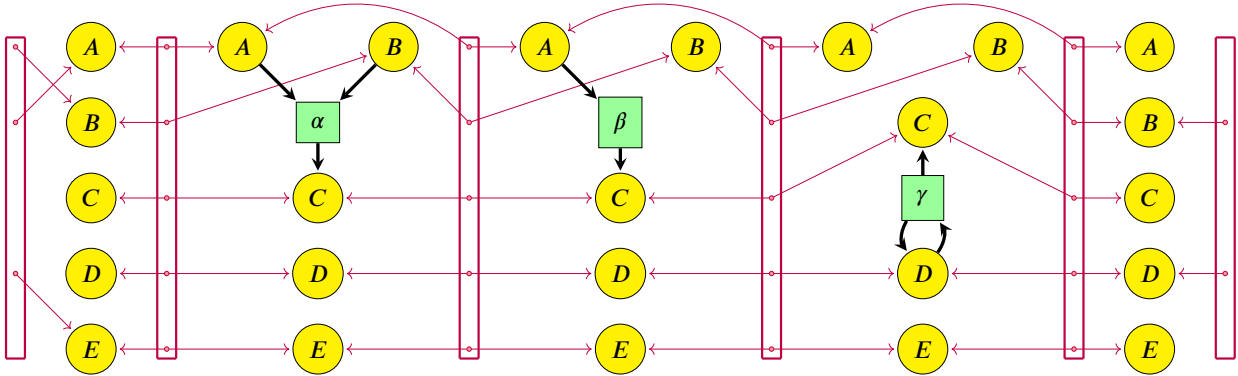

Next we state and prove our secound decomposition Lemma of open Petri nets into $P_{n,m}$ factors.

\begin{lemma}
\label{lemma:decomp}
Any open Petri net $\mathcal{P}$ can be factored as 
\[
    \mathcal{Q}\cp\left(\mathcal{G}_0 \oplus \mathcal{G}_{1}\oplus \mathcal{G}_{2}\oplus \cdots \oplus \mathcal{G}_{N}\right)\cp \mathcal{Q}',
\] where:
\begin{itemize}
    \item $\mathcal{Q}$ and $\mathcal{Q}'$ are transitionless open Petri nets,
    \item $\mathcal{G}_0$ is an identity morphism in $\openpetri$, and
    \item For $i = 1,\cdots, N$, $\mathcal{G}_{i}$ is a body net decorated with an atomic Petri net isomorphic to $P_{m_i,n_i}$.
\end{itemize} 
\end{lemma}
\begin{proof}

First, we define the Petri nets that will be involved in our composite. 

For each transition $\tau \in T$ we will define an open Petri net $G_\tau$ which has a unique input species for each input arc incoming to $\tau$ and a unique output species for each output arc outgoing from $\tau$. In particular, the sets 
\[
    I_\tau \coloneqq \sum_{\sigma \in S} s(\tau)(\sigma),\quad 
    O_\tau \coloneqq \sum_{\sigma \in S} t(\tau)(\sigma),
\] 
represent the input and output arcs to the transition $\tau$. Define $S_\tau \coloneqq I_\tau + O_\tau$. Define the Petri net $G_\tau$ to have a single transition and species $S_\tau$. Its source and target maps are defined so that there is a single input arc from each species in $I_\tau$ to the transition and there is a single output arc from the transition to each species in $O_\tau$. Note that $G_\tau$ is isomorphic to $P_{|I_\tau|, |O_\tau|}$.  Let $\Ga_\tau$ be the open Petri net 
\[
    \left(S_\tau \xrightarrow{1_{S_\tau}} S_\tau \xleftarrow{1_{S_\tau}} S_\tau, G_\tau = \left(S_\tau, 1, s_\tau, t_\tau\right)\right).
\]

Next, let $S_0$ be the species in $S$ which are neither the input species nor the output species of any transition.  Let $\Ga_0$ be the identity morphism on $S_0$.

Consider the monoidal product $\Ga = \Ga_0 \oplus \left(\bigoplus_{\tau \in T} \Ga_\tau\right)$. 
Explicitly, 
\[
    \Ga = \left(\tilde S \xrightarrow{1_{\tilde S}} \tilde S \xleftarrow{1_{\tilde S}} \tilde S, \left(G = \tilde S, T, \tilde s, \tilde t\right)\right),
\] where  $\tilde S = S_0 + \sum_{\tau \in T} S_\tau$.

Next we want to create transitionless open Petri nets $\Qa$ and $\Qa'$ which glue together the species in $G$ which arose from the same species in the original Petri net $P$. We begin by defining a map $\tilde h \colon \tilde S \to S$ which maps each species of $G$ to the species in $P$ to which it corresponds. 

For each transition $\tau$, let $f_\tau\colon I_\tau \to S$ be the unique map such that for each $\sigma \in S$, $f_\tau$ is the constant $\sigma$ map on $s(\tau)(\sigma)$. Likewise, let $g_\tau\colon O_\tau \to S$ be the unique map such that for each $\sigma \in S$, $g_\tau$ is the constant $\sigma$ map on $t(\tau)(\sigma)$. Together these induce a map $h_\tau \coloneqq [f_\tau, g_\tau]\colon S_\tau \to S$. Let $h_0$ be the inclusion of $S_0$ into $S$. Then let $\tilde h\colon \tilde S \to S$ be the universal morphism induced by $h_0$ and $h_\tau$ for $\tau \in T$.

Finally, we define 
\[
    \Qa \coloneqq  \left(X \xrightarrow{i} S \xleftarrow{h} \tilde S, 0_S\right), \quad
    \Qa' \coloneqq \left(\tilde S \xrightarrow{h} S \xleftarrow{o} Y, 0_S\right).
\]
We will show that $\Pa = \Qa \cp \Ga \cp \Qa'$.
First consider $\Qa \cp \Ga$. The composition of the underlying cospans is as follows: 
      \[\begin{tikzcd}
    	&& S \\
    	& S && {\tilde S} \\
    	X && {\tilde S} && {\tilde S}
    	\arrow["h"', from=3-3, to=2-2]
    	\arrow["i", from=3-1, to=2-2]
    	\arrow["{1_{\tilde S}}", from=3-3, to=2-4]
    	\arrow["{1_{\tilde S}}"', from=3-5, to=2-4]
    	\arrow["{1_S}", dashed, from=2-2, to=1-3]
    	\arrow["h"', dashed, from=2-4, to=1-3]
           \arrow["\lrcorner"{anchor=center, pos=.85, rotate=-45, scale = 2}, draw=none, from=3-3, to=1-3]
       \end{tikzcd}\]

The decoration of the composite is $(S, T, s', t')$ where $s'\colon T \to \Nb^S$ is defined by 
\begin{align}
    s'(\tau)(\sigma) & = \sum_{\tilde \sigma \in \tilde S | h(\tilde \sigma) = \sigma} \tilde s(\tau)(\tilde \sigma) \nonumber \\
    & = \sum_{\tilde \sigma \in S_0 | h_0(\tilde \sigma) = \sigma} + \sum_{\upsilon \in T} \sum_{\tilde \sigma \in S_\upsilon | h_\upsilon(\tilde \sigma) = \sigma} \tilde s(\tau)(\tilde \sigma) \nonumber \\
    & = \sum_{\tilde\sigma \in S_\tau | h_\tau(\tilde \sigma) = \sigma} \tilde s(\tau)(\tilde \sigma) \label{eq:reduce1}\\
    & = \sum_{\tilde \sigma \in I_\tau | f_\tau(\tilde \sigma) = \sigma} s_\tau(\tilde \sigma) + \sum_{\tilde \sigma \in O_\tau | g_\tau (\tilde \sigma) = \sigma} s_\tau(\tilde \sigma) \nonumber \\
    & = \sum_{s(\tau)(\sigma)} 1 + \sum_{t(\tau)(\sigma)} 0 \label{eq:reduce2} \\
    & = s(\tau)(\sigma). \nonumber
 \end{align}
 
 Note that Equation~\ref{eq:reduce1} follows from the fact that for $\tilde s(\tau)(\tilde \sigma)$ is $0$ unless $\tilde \sigma \in S_\tau$. Equation~\ref{eq:reduce2} follows from the fact that $f_\tau$ is the constant $\tilde \sigma$ map on $s(\tau)(\tilde \sigma)$, and thus $f_\tau(\tilde \sigma) = \sigma$ if and only if $\tilde \sigma \in s(\tau)(\sigma)$.
 
 The above calculations prove that $s' = s$. An identical argument shows that $t' = t$ as well. Therefore the composite $\Qa \cp \Ga$ is decorated by the original Petri net $P$. 
 
 Finally, we must show that $(\Qa \cp \Ga) \cp  \Qa'$ is $\Pa$. First we examine the composite of their underlying cospans:

    \[\begin{tikzcd}
    	&& S \\
    	& S && S \\
    	X && {\tilde S} && Y
    	\arrow["h"', from=3-3, to=2-2]
    	\arrow["i", from=3-1, to=2-2]
    	\arrow["h", from=3-3, to=2-4]
    	\arrow["o"', from=3-5, to=2-4]
    	\arrow["{1_S}", dashed, from=2-2, to=1-3]
    	\arrow["{1_S}"', dashed, from=2-4, to=1-3]
           \arrow["\lrcorner"{anchor=center, pos=.85, rotate=-45, scale = 2}, draw=none, from=3-3, to=1-3]
       \end{tikzcd}\]

Note that $S+_{\tilde S} S = S$ because $h\colon \tilde S \to S$ is surjective. That the decoration of the composite is $P$ again is a straightforward application of the definition of composing open Petri nets. 
\end{proof}

\begin{example}
Consider the open Petri net depicted in Figure~\ref{fig:example-petri}. 

Applying the Lemma~\ref{lemma:decomp} this Petri net can be decomposed as 
\[
    \Qa \cp (\Ga_0 \oplus \Ga_1 \oplus \Ga_2 \oplus \Ga_3) \cp \Qa'
\] where $\Ga_0$ is decorated with a transitionless Petri net and $\Ga_1, \Ga_2, \Ga_3$ are decorated with the atomic Petri nets.
Here, $\alpha$ has two input arcs and one output arc, so the atomic Petri net $P_{2,1}$ decorates $\Ga_1$;
likewise, $\beta$ has one input arc and one output arc, so $P_{1,1}$ decorates $\Ga_2$;
and $\gamma$ has one input arc and two output arcs, so $P_{1,2}$ decorates $\Ga_3$.
The decomposition is depicted in Figure~\ref{fig:example-decomposition}. 
\end{example}

\begin{figure}[h]
    \centering
    \[
        \begin{tikzpicture}
	\begin{pgfonlayer}{nodelayer}
		\node [style=species, inner sep= 2mm] (A) at (-3.5, 1.5) {$A$};
		\node [style=species, inner sep= 2mm] (B) at (0.5, 1.5) {$B$};
		\node [style=species, inner sep= 2mm] (C) at (-1.5, -0.25) {$C$};
		\node [style=species, inner sep= 2mm] (D) at (2.5, -0.25) {$D$};
		\node [style=species, inner sep= 2mm] (E) at (-1.5,-1.5) {$E$};
        
        \node [style=transition,  inner sep= 2mm, fill=green!40] (a) at (-1.5, 1.5) {$\alpha$};
        \node [style=transition,  inner sep= 2mm, fill=green!40] (b) at (-3.5, -0.25) {$\beta$};
        \node [style=transition,  inner sep= 2mm, fill=green!40] (c) at (0.5, -0.25) {$\gamma$}; 
		
		\node [style=empty] (X) at (-5.1, 1) {${\color{gray}X}$};
		\node [style=none] (Xtr) at (-4.75, 0.75) {};
		\node [style=none] (Xbr) at (-4.75, -1.25) {};
		\node [style=none] (Xtl) at (-5.4, 0.75) {};
        \node [style=none] (Xbl) at (-5.4, -1.25) {};
	
		\node [style=inputdot, fill = red!40, inner sep= 1mm] (1) at (-5, 0.5) {};
		\node [style=empty] at (-5.2, 0.5) {$1$};
		\node [style=inputdot, fill = red!40, inner sep= 1mm] (2) at (-5, 0) {};
		\node [style=empty] at (-5.2, 0) {$2$};
		\node [style=inputdot, fill = red!40, inner sep= 1mm] (3) at (-5, -0.5) {};
		\node [style=empty] at (-5.2, -0.5) {$3$};
		\node [style=inputdot, fill = red!40, inner sep= 1mm] (4) at (-5, -1) {};
		\node [style=empty] at (-5.2, -1) {$4$};

		\node [style=empty] (Y) at (4.1, 1) {${\color{gray}Y}$};
		\node [style=none] (Ytr) at (4.4, 0.75) {};
		\node [style=none] (Ytl) at (3.75, 0.75) {};
		\node [style=none] (Ybr) at (4.4, -0.75) {};
		\node [style=none] (Ybl) at (3.75, -0.75) {};

		\node [style=inputdot, fill = red!40, inner sep= 1mm] (5) at (4, 0.5) {};
		\node [style=empty] at (4.2, 0.5) {$5$};
		\node [style=inputdot, fill = red!40, inner sep= 1mm] (6) at (4, -0.5) {};
		\node [style=empty] at (4.2, -0.5) {$6$};		
		
		
	\end{pgfonlayer}
	\begin{pgfonlayer}{edgelayer}
		\draw [style=inarrow] (A) to (a);
        \draw [style=inarrow] (A) to (b);
		\draw [style=inarrow] (B) to (a);
        \draw [style=inarrow,bend left=30, looseness=1.00] (D) to (c);
		\draw [style=inarrow] (a) to (C);
        \draw [style=inarrow] (b) to (C);
        \draw [style=inarrow] (c) to (C);
        \draw [style=inarrow,bend left=30, looseness=1.00] (c) to (D);

		\draw [style=inputarrow,bend left=60, looseness=1.25] (1) to (B);
		\draw [style=inputarrow] (2) to (A);
		\draw [style=inputarrow] (3) to (A);
		\draw [style=inputarrow] (4) to (E);
		\draw [style=inputarrow] (5) to (B);
        \draw [style=inputarrow] (6) to (D);
		\draw [style=simple] (Xtl.center) to (Xtr.center);
		\draw [style=simple] (Xtr.center) to (Xbr.center);
		\draw [style=simple] (Xbr.center) to (Xbl.center);
		\draw [style=simple] (Xbl.center) to (Xtl.center);
		\draw [style=simple] (Ytl.center) to (Ytr.center);
		\draw [style=simple] (Ytr.center) to (Ybr.center);
		\draw [style=simple] (Ybr.center) to (Ybl.center);
		\draw [style=simple] (Ybl.center) to (Ytl.center);
	\end{pgfonlayer}
\end{tikzpicture}
    \]
    \caption{An open Petri net with three transitions.}
    \label{fig:example-petri}
\end{figure}
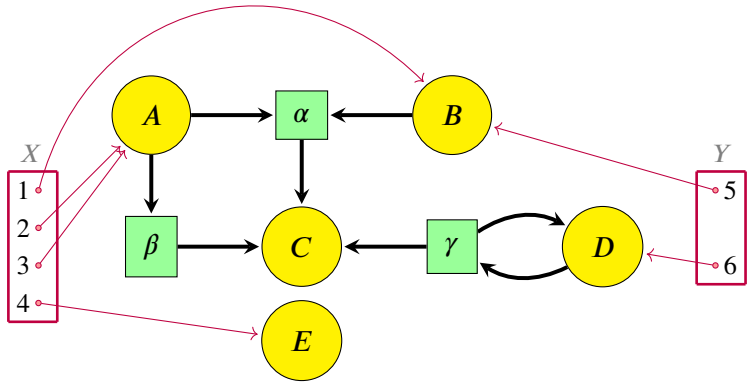

\begin{center}
\includegraphics[scale=1]
{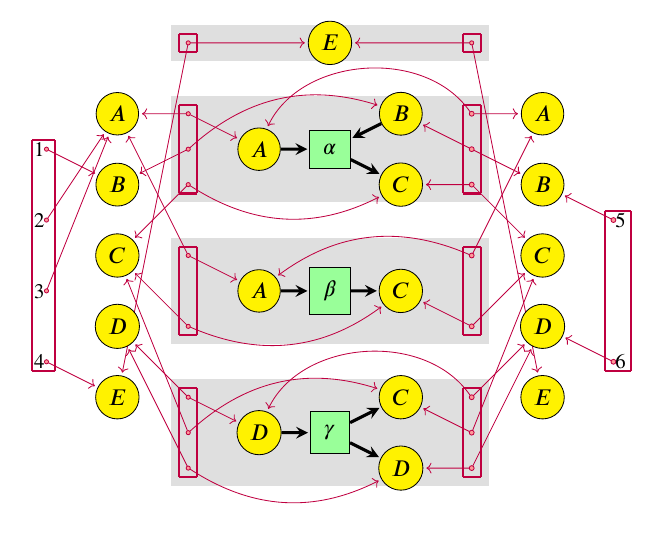}
\captionof{figure}{The decomposition of the open Petri net depicted in Figure~\ref{fig:example-petri} into atomic Petri nets as defined by Lemma~\ref{lemma:decomp}. The decomposition is $\Qa \cp (\Ga_0 \oplus \Ga_1 \oplus \Ga_2 \oplus \Ga_3) \cp \Qa'$. Graphically each $\Ga_i$ is enclosed by a grey box and they are shown in numerical order from top to bottom.}
    \label{fig:example-decomposition}
\end{center}

\section{Additive Invariants of Open Petri Nets}\label{sec:invariants}
This section is devoted to the classification theorems of additive invariants for open Petri nets and monically open Petri nets. These are discussed in Sections~\ref{sec:invariants-open} and~\ref{sec:invariants-monic-open}, respectively.
We begin by establishing that all functors from $\openpetri$ to $\cc{B}\mathbb{N}$ are monoidal.
First a definition.

\begin{definition}\label{def:additive-invariant}
An \defined{additive invariant of open Petri nets} is a monoidal functor $\openpetri \to \cc{B}\mathbb{N}$ where $\cc{B}\mathbb{N}$ is the one-object category induced by the monoid $(\mathbb{N}, +)$. The monoidal product of $\bn$ is also given by $+$.
\end{definition}

\subsection{Every invariant of \(\openpetri\) is additive}\label{sec:proof:thm:all-invariants-are-additive}

Before proving our main classification theorems, we 
first prove that all functors \(\openpetri \to \mathbf{B}\mathbb{N}\) are in fact monoidal. Therefore our classification of monoidal functors \(\openpetri \to \mathbf{B}\mathbb{N}\) is in fact a classification of \emph{all} functors \(\openpetri \to \mathbf{B}\mathbb{N}\). \\

We remind the reader of the Eckmann-Hilton argument, which states that in any monoidal category, composition and the monoidal product coincide when restricted to endomorphisms of the unit object. 
As a consequence, for open Petri nets with empty interfaces, composition coincides with their direct sum.
Specifically, if an open Petri net \( \Qa \) has an empty codomain, then:
\(
\Qa \circ \id_{\emptyset} = \Qa = \Qa \oplus \id_{\emptyset}.
\)
Likewise if \( \Qa' \) has an empty domain. Then:
\[
\Qa \circ \Qa' = (\Qa \oplus \id_{\emptyset}) \circ (\id_{\emptyset} \oplus \Qa') 
= (\Qa \circ \id_{\emptyset}) \oplus (\id_{\emptyset} \circ \Qa') 
= \Qa \oplus \Qa'.
\]
This follows naturally from the interchange law of monoidal categories.
With this in mind, we begin with the following Lemma.

\begin{lemma}\label{lem:transitionless} 
Let $F \colon \openpetri\to \bn$ be a functor, and let $\mathcal{P}$ be an open Petri net. If $\Pa$ is transitionless, then $F(\mathcal{P})=0$.

\end{lemma}
\begin{proof}
By Lemma~\ref{lem:generating-transitionless-petris}, any transitionless open Petri nets can be generated by the open Petri nets $\mu$, $\eta$, $\delta$, and $\epsilon$. Therefore it is sufficient to show that $F$ is trivial on these open Petri nets. 

Since $F$ is functorial, $F(\id_S) = 0$ for all species  set $S$. Thus, $\delta \cp \mu = \id_1$ implies that 
\[
    0 = F(\delta \cp \mu) = F(\delta) + F(\mu).
\] Since $F(\delta)$ and $F(\mu)$ are natural numbers, this implies that $F(\delta) = 0$ and $F(\mu) = 0$.

It remains to show that $F$ vanishes on $\eta = \eta_1$ and $\epsilon = \epsilon_1$ as well.  From the standard relations $\eta_1 = \eta_2 \cp \mu $ and $\epsilon_1 = \delta \cp \epsilon_2 $ as in \ref{rem01}, we obtain
\begin{align*}
    F(\eta_1\cp \epsilon_1) &= F(\eta_2) + F(\mu) + F(\delta) + F(\epsilon_2)\\
    & = F(\eta_2) + F(\epsilon_2)\\
    &= F(\eta_2 \cp \epsilon_2).
\end{align*}

\noindent
As composition and direct sum coincide when both open Petri nets have empty domain and codomain,
and since \( \eta_m \cp \epsilon_m \) decomposes as the direct sum
\(
\eta_m \cp \epsilon_m = \bigoplus_{i=1}^{m} (\eta_1 \cp \epsilon_1),
\)
it follows that
\[
F(\eta_2 \cp \epsilon_2) = 2 F(\eta_1 \cp \epsilon_1).
\]
Since $(\mathbb{N},+)$ is a cancellative monoid, this implies $0=F(\eta_1 \cp \epsilon_1) = F(\eta_1) + F(\epsilon_1)$. Therefore, $F(\eta) =0$ and $ F(\epsilon) = 0$, as required.
\end{proof}

\begin{remark}
Notice that the result above holds for any cancellative monoid where the identity is the only element with an inverse.
Further, the classification of invariants as a linear combination of generators relies on the monoid being abelian and cancellative.
\end{remark}

\begin{theorem}\label{thm:all-invariants-are-additive}
Every functor \(\openpetri \to \mathbf{B}\mathbb{N}\) is monoidal.
\end{theorem}
\begin{proof}

\newcommand{\QsY}{Y\pr}
Let $\Pa$ and $\Qa$ be the open Petri nets 
    \begin{align*}
    \Pa &= \left(X \xrightarrow{i} S \xleftarrow{o} Y, \; \left(s, t\colon T \to \mathbb{N}^S\right) \right) \\
    \Qa &= \left(U \xrightarrow{i\pr} S\pr \xleftarrow{o\pr} V, \; \left(s\pr, t\pr\colon T\pr \to \mathbb{N}^{S\pr}\right) \right).
    \end{align*}  
    Then \[
        \Pa \oplus \Qa = (\Pa \cp \id_{Y}) \oplus (\id_{U} \cp Q) = (\Pa \oplus \id_{U}) \cp (\id_{Y} \oplus Q).
    \]
    Therefore, it suffices to show for each open Petri net $\Pa$ and each finite set $M$ that $F(\Pa \oplus \id_M) = F(\Pa)$.

    We begin by considering the open Petri net 
    \[
        \Ra \coloneqq (\eta_{X} \oplus\eta_M) \cp (\Pa \oplus \id_M ) \cp (\epsilon_{Y} \oplus\epsilon_M),
    \] 
    whose domain and codomain are both $\emptyset$.   
Since composition and direct sum coincide in such cases and using the interchange law, we have
    \begin{align*}
         \Ra &= (\eta_{X} \cp \Pa \cp \epsilon_{Y}) \oplus (\eta_M \cp \id_M \cp \epsilon_M),\\
         & = (\eta_{X} \cp \Pa \cp \epsilon_{Y}) \cp (\eta_M \cp \id_M \cp \epsilon_M).
    \end{align*}
     
It follows that,
    \begin{equation*}\label{eq:thm1.6}
      F(\eta_{X + M}) + F(\Pa\oplus \id_M) + F(\epsilon_{Y + M}) = F(\Ra) = F(\eta_{X}) + F(\Pa) + F(\epsilon_{Y}) + F(\eta_{M} ) + F( \epsilon_{M}).
    \end{equation*}
    This derivation uses that $\eta_{X} \oplus \eta_{M} = \eta_{X+M}$, and likewise that $\epsilon_{Y} \oplus \epsilon_{M} = \epsilon_{Y+M}$.

    Finally, by Lemma \ref{lem:transitionless}, all of the boundary terms are $0$, revealing that
    $F(\Pa \oplus \id_M) = F(\Pa)$.
    
\end{proof}

\subsection{Classifying additive invariants of open Petri nets}\label{sec:invariants-open}

We begin by proving the functoriality of the maps $F_{m,n}\colon \openpetri \to \bn$ that take an open Petri net to the number of transitions having $m$ input arcs and $n$ output arcs. These will form the building blocks of our classification of additive invariants.

\begin{lemma}\label{lemma:Fmn-functorial}
There is a monoidal functor $F_{m,n}\colon \openpetri \to \bn$ such that for an open Petri net $\mathcal P$ with decoration $P$, $F_{m,n}(\mathcal P)$ is the number of transitions in $P$ with $m$ input arcs and $n$ output arcs. 
\end{lemma}
\begin{proof}
For an object $S$ in $\openpetri$, the identity open Petri net is $\id_S = \left(S \xrightarrow{1_S} S \xleftarrow{1_S}S, 0_S\right)$. Since $0_S$ has no transitions, we obtain $F_{m,n}\left(S \xrightarrow{1_S} S \xleftarrow{1_S}S, 0_S\right) = 0$.

Consider two open Petri nets $\Pa = (X \to S \leftarrow Y, P)$ and $\Pa' = (Y \to S' \leftarrow Z, P')$, such that $P$ and $P'$ have $T$ and $T'$ as transition sets, respectively. Recall that  the Petri net decorating $\Pa \cp \Pa'$ has transition set $T + T'$. For $\tau \in T$, suppose that $\tau$ has $k$ input arcs in $P$. Equations~\ref{eq:input-arcs} and~\ref{eq:composite-source} imply that $\tau$ has $k$ input arcs in Petri net decorating $\Pa' \cp \Pa$. Likewise for output arcs and for $\tau' \in T'$. Therefore, $F_{m,n}(\Pa\cp  \Pa') =  F_{m,n}(\Pa') + F_{m,n}(\Pa)$. Similarly, for open Petri nets $\Pa$ and $\Qa$, we have $F_{m,n}(\Pa \oplus \Qa) =  F_{m,n}(\Pa) + F_{m,n}(\Qa)$.
\end{proof}

\begin{example}
    Consider the open Petri net in Figure~\ref{fig:example-petri}.
    The transition labeled $\alpha$ has two input arcs and one output arc, so $F_{2,1}$ applied to this open Petri net is $1$. The transition $\beta$ has one input arc and one output arc. Therefore, $F_{1,1}$ applied to this open Petri net is $1$. Finally, the transition $\gamma$ has one input arc and two output arcs; $F_{1,2}$ applied to this open Petri net is $1$.  For all other $m,n$, $F_{m,n}$ applied to these open Petri nets is $0$.
\end{example}

We are now ready to present the proof of one of our main results, Theorem~\ref{thm:open-petri-main-thm}.

\begin{theorem}\label{thm:open-petri-main-thm}
For any two $m, n \in \mathbb{N}$, there is a functor \[F_{m,n}\colon \openpetri \to \bn\] which maps each open Petri net to the total number of transitions with exactly $n$ input arcs and $m$ output arcs. Furthermore, any additive invariant $G\colon \openpetri \to \bn$ is completely determined --- as a linear combination of the functors $F_{m,n}$ defined above --- by its values on the family $\Pa_{m,n}$. In particular, we have 
\[
    G(-) = \sum_{m,n \in \mathbb{N}} G(\Pa_{m,n}) F_{m,n}(-).
\]
\end{theorem}
\begin{proof}

In Lemma~\ref{lemma:Fmn-functorial} we verified that the maps $F_{m,n}\colon \openpetri \to \bn$ as defined in the first sentence of the Theorem statement are indeed monoidal functors. Next we show that any monoidal functor $F \colon \openpetri \to \bn$ satisfies
\[
    F = \sum_{m,n \in \Nb} F(\Pa_{m,n}) F_{m,n}.
\]

First, for the open Petri net $\Pa_{\bar m,\bar n}$ observe that $\sum_{m,n \in \Nb} F(\Pa_{m,n}) F_{m,n}(\Pa_{\bar m, \bar n}) = F(\Pa_{\bar m, \bar n})$ because $F_{m,n}(\Pa_{\bar m, \bar n})$ is $1$ if $m = \bar m$ and $ n = \bar n$ and $0$ otherwise.

Now let $\mathcal{P} = (X\xrightarrow{} S \xleftarrow{} Y, P)$ be any open Petri net. Let $\mathcal{Q} \cp \left( \bigoplus_{i=0}^{N} \mathcal{G}_{i} \right) \cp \mathcal{Q}'$ be the decomposition of $\Pa$ as given in Lemma~\ref{lemma:decomp}. Thus, using this decomposition of $\Pa$, we can invoke the additivity of $F$ together with Lemma~\ref{lem:transitionless} to obtain 
\[
    F(\Pa)  = F(\Qa) + F(\Ga_0) + \sum_{i = 1}^N F(\Ga_i) + F(\Qa') = \sum_{i = 1}^N F(\Ga_i).
\]   
Likewise $\sum_{m,n} F(\Pa_{m,n})F_{m,n}$ is an additive invariant and so 
\[
    \sum_{m,n} F(\Pa_{m,n})F_{m,n}(\Pa) = \sum_{i = 1}^N \sum_{m,n} F(\Pa_{m,n})F_{m,n}(\Ga_i)
\]

For $i = 1,..., N$, $\Ga_{i}$ is decorated with an atomic Petri net $P_{m_i, n_i}$. Therefore, $\Ga_i$ is isomorphic to $\Pa_{m_i, n_i}$ and hence $F(\Ga_i) = \sum_{m,n} F(\Pa_{m,n})F_{m,n}(\Ga_i)$. 
Thus, as desired, we have that
\[
     \sum_{m,n} F(\Pa_{m,n})F_{m,n}(\Pa) = \sum_{i = 1}^N \sum_{m,n} F(\Pa_{m,n})F_{m,n}(\Ga_i) = \sum_{i = 1}^N F(\Ga_i) = F(\Pa) .
\]

\end{proof}

\subsection{Classifying additive invariants of monically open Petri nets}\label{sec:invariants-monic-open}

In this section, we show that the invariants of mope nets are linear combinations of some
countable generating sets. These play the role of $F_{m,n}$ in the classification of additive invariants for all open Petri nets.  In this case, the generators correspond to the two subclasses of mope nets which we have already encountered, namely {\it body  nets} and {\it boundary  nets}, which were defined in Definitions~\ref{def:body-type} and~\ref{def:boundary-nets}, respectively.

Our first set of generating functors consist of functors $F_\bt$ for body types $\bt$. These are defined in Lemma~\ref{lem:body-type-functor} and capture invariants of body nets. The second set of generating functors consist of functors $F_{a,z}$ where $\{a_k\}_{k \in \Nb}$ and $\{z_k\}_{k \in \Nb}$ are non-decreasing sequences satisfying the law in Equation~\ref{eq:az-condition}. These are   defined in Lemma~\ref{lem:boundary-type-functor} and capture invariants having to do with the underlying cospan of the open Petri net.

\begin{lemma}\label{lem:body-type-functor}
Let $\bt$ be a body type. There is a monoidal functor $F_\bt \colon \mope \to \bn$ such that for a mope net $\Ma$ with decoration $P$, $F_\bt(\Ma)$ is the number of transitions in $P$ with type $\bt$.
\end{lemma}
\begin{proof}
That $F_\bt$ is $0$ on the identity and respects the monoidal product follows the same reasoning as in the proof of Lemma~\ref{lemma:Fmn-functorial}. 

Consider two mope nets $\Ma = (X \to S \leftarrow Y, P)$ and $\Ma' = (Y \to S' \leftarrow Z, P')$ such that $P$ and $P'$ have $T$ and $T'$ as transition sets, respectively. Recall that the Petri net decorating $\Ma \cp \Ma'$ has transition set $T + T'$. For $\tau \in T$, suppose that $\tau$ has body type $\bt$. Since the legs of $\Ma$ and $\Ma'$ are monic, no two input and/or output species to $\tau$ are identified. Thus, $\tau$ also has body type $\bt$ in the Petri net decorating $\Ma \cp \Ma'$. Likewise if $\tau$ does not have body type $\bt$ in the Petri net decorating $\Ma$, then $\tau$ does not have body type $\bt$ in the Petri net decorating $\Ma \cp \Ma'$. The same is true for transitions in the Petri net decorating $\Ma'$. Therefore $F_\bt(\Ma \cp \Ma') = F_\bt(\Ma) + F_\bt(\Ma')$.
\end{proof}

\begin{lemma}\label{lem:boundary-type-functor}
    Let $\{a_k\}_{k \in \Nb}$ and $\{z_k\}_{k \in \Nb}$ be non-decreasing sequences of natural numbers that satisfy 
    \begin{equation} \label{eq:az-condition}
    k(a_1 + z_1) = a_k + z_k
    \end{equation} for all $k \in \Nb$. Then there is a functor $F_{a,z} \colon  \mope \to \bn$ defined such that for a mope net $\Ma  = (X \to S \leftarrow Y, P)$  
    \[
        F_{a,z}(\Ma) = (a_{|S|} - a_{|X|}) + (z_{|S|} - z_{|Y|}).
    \]
\end{lemma}
\begin{proof}

First, $F_{a,z}$ is not only integer-valued but
in fact natural-valued, since $a,z$ are
non-decreasing.
And
by direct computation $F_{a,z}(\id_S) = (a_{|S|}-a_{|S|}) + (z_{|S|}-z_{|S|}) = 0$.

We now verify that $F_{a,z}$ respects composition. Consider two mope nets $\Ma = (X \to S \leftarrow Y, P)$ and $\Ma' = (Y \to S' \leftarrow Z, P')$.
By the monotonicity of the legs of $\Ma$ and $\Ma'$, there are $|S| + |S' | - |Y|$ species in the decoration of $\Ma \cp \Ma'$. Therefore,

\begin{align*}
  F(\Ma) + F(\Ma') & = (a_{|S|}-a_{|X|} + z_{|S|}-z_{|Y|}) + (a_{|S'|}-a_{|Y|} + z_{|S'|}-z_{|Z|})
  \\&=~
  (|S|+|S\pr|-|Y|)(a_{1} + z_{1})-a_{|X|}-z_{|Z|}
  \\&=~
  (a_{|S|+|S\pr|-|Y|} + z_{|S|+|S\pr|-|Y|})-a_{|X|}-z_{|Z|}
  \\&=~
  (a_{|S| + |S'| - |Y|}-a_{|X|}) + (z_{|S| + |S'| - |Y|}-z_{|Z|})
  \\&=~
  F(\Ma \cp \Ma').
\end{align*}
Hence, $F_{a,z}$ is functorial.
\end{proof}

We next turn to show that any functor must satisfy a certain law and the value of this functor on transitionless mope net is determined by its value on boundary mope nets.

\begin{lemma}\label{lem:functor-implies-condition}
Let $F \colon \mope \to \bn$ be a functor and let $\Ma = (X \to S \leftarrow Y, 0_S)$ be a transitionless mope net. Then the following holds 

\begin{enumerate}
    \item[(a)] $k (F(\eta_1)+F(\epsilon_1)) = F(\eta_k)+F(\epsilon_k)$ for all natural $k$.
    \item[(b)] $F(\Ma) = (F(\eta_{S}) - F(\eta_{X})) + (F(\epsilon_{S}) - F(\epsilon_{Y})).$
\end{enumerate}
\end{lemma}

\begin{proof}
The composite mope net
\begin{equation*}
    \eta_{|X|}\cp\Ma\,\cp\,\epsilon_{|Y|} = (0 \to S \leftarrow 0, 0_S)
\end{equation*}
is isomorphic to $(\eta_1 \cp \epsilon_1)$ composed with itself $|S|$ times. Therefore,
\begin{equation}\label{eq:F-on-boundary-nets}
    |S|(F(\eta_1)+ F(\epsilon_1)) = F(\eta_{X})+F(\Ma)+F(\epsilon_{Y}).
\end{equation}

Applying Equation~\ref{eq:F-on-boundary-nets} to the case where $\Ma$ is the identity mope net on $S$ and $|S| = k$ proves the first assertion that $k
(F(\eta_1)+F(\epsilon_1)) = F(\eta_k)+F(\epsilon_k)$.

Finally to prove (b), applying the equality $|S|(F(\eta_1) + F(\epsilon_1)) = F(\eta_{S}) + F(\epsilon_{S})$ to Equation~\ref{eq:F-on-boundary-nets} and rearranging implies that $F(\Ma) = (F(\eta_{S}) - F(\eta_{X})) + (F(\epsilon_{S}) - F(\epsilon_{Y}))$.
\end{proof}

We will next prove a crucial result that is similar in nature to Theorem~\ref{thm:all-invariants-are-additive}. Note that the proof technique in Theorem~\ref{thm:all-invariants-are-additive} (specifically Lemma~\ref{lem:transitionless}) cannot be extended to mope nets since it requires the use of the non-monic open Petri nets \(\mu\) and \(\delta\).

\begin{lemma}\label{lem:moo}
  Let 
  $F \colon  \mope \to  \bn$ be
  a functor. Let
  $\mathcal{M}$ be any endomorphism in $\mope$ and let $\id_{S}$ be the identity mope net with $S$ species.
  Then
  $F(\mathcal{M} \oplus \id_{S}) = F(\mathcal{M})$.
\end{lemma}
\begin{proof}
  Let $\Ma$ be a mope net with domain and codomain $X$. Define $$\mathcal{P} \coloneqq (\eta_{X}\oplus\eta_{S}) \cp (\mathcal{M} \oplus \id_{S}) \cp (\epsilon_{X}\oplus\epsilon_S)$$
  Note that $\eta_{X} \oplus \eta_{S} = \eta_{X + S}$ and likewise $\epsilon_{X} \oplus \epsilon_{S} = \epsilon_{X + S}$. Applying $F$ shows that
  \begin{equation}\label{eq:F-applied-to-P-1}
      F(\mathcal{P}) = F(\eta_{|X + S|}) + F(\mathcal{M} \oplus
  \id_{S}) + F(\epsilon_{|X + S|}) = |X + S| (F(\eta_1) + F(\epsilon_1)) + F(\mathcal{M} \oplus
  \id_{S}).
  \end{equation} 
  The second equality follows from the equality proven in Lemma~\ref{lem:functor-implies-condition}.
  
  On the other hand, rearranging the factors of $\Pa$ yields
  $$\mathcal{P} = (\eta_{X} \cp \mathcal{M}\cp \epsilon_{X}) \oplus (\eta_{S} \cp \id_{S} \cp \epsilon_{S}).$$
  The monoidal product of two open Petri nets whose feet are both the empty set is isomorphic to their composite. Since the two factors of $\Pa$ satisfy this criteria, we have
  $$\mathcal{P} = (\eta_{X} \cp \mathcal{M}\cp \epsilon_{X}) \cp (\eta_{S} \cp \id_{S} \cp \epsilon_{S})$$
  
  Applying $F$ to this equality yields
  \begin{equation}\label{eq:F-applied-to-P-2}
    F(\mathcal{P}) =  F(\eta_{X}) + F(\mathcal{M}) + F(\epsilon_{X}) + F(\eta_{X}) + F(\id_{S}) + F(\epsilon_{S}) = |X + S| (F(\eta_1) + F(\epsilon_1)) + F(\mathcal{M}).
  \end{equation} Again, the second equality follows from applying the equality proven in Lemma~\ref{lem:functor-implies-condition} twice.

  Comparing Equation~\ref{eq:F-applied-to-P-1} and Equation~\ref{eq:F-applied-to-P-2} and cancelling in $\Nn$, we establish the claim.
\end{proof}

We are now ready for the main result of this section.

\begin{theorem}\label{thm:mope-invariant-classification}
   Every functor $F \colon  \mope \to  \bn$ decomposes as  
    \[
        F_{a,z} + \sum_{\bt} d_\bt F_\bt
    \] as $\bt$ ranges over body types for coefficients $d_\bt \in \Nb$ and non-decreasing sequences $\{a_k\}_{k \in \Nb}$ and $\{z_k\}_{k \in \Nb}$ satisfying the condition in Equation~\ref{eq:az-condition}.
\end{theorem}
\begin{proof}
The functors have been introduced in Lemma~\ref{lem:body-type-functor} and Lemma~\ref{lem:boundary-type-functor}.

Now let $F\colon \mope \to \bn$ be any 
functor. Define sequences $a_k = F(\eta_k)$ and $z_k = F(\epsilon_k)$ along with coefficients $d_\bt = F(\Pa_\bt)$.

First we show that the sequences $a$ and $z$ are non-decreasing. Note that \[\eta_{k + 1} = \eta_k \cp (k \to {k + 1} \leftarrow k+1, 0_{k + 1}).\] Applying $F$ to both sides and appealing to functoriality implies that $a_{k +1} = F(\eta_{k + 1})$ is greater than or equal to $a_k = F(\eta_k)$. Therefore, the sequence $a$ is non-decreasing and likewise for $z$. Furthermore, they also  satisfy the condition in Equation~\ref{eq:az-condition} by Lemma~\ref{lem:functor-implies-condition}. 

Let $\Ma = (X \to S \leftarrow Y, P)$ be any mope net. By Corollary~\ref{lem:extract-net}, $\Ma = \Qa \cp \Ga_1 \cp \cdots \cp \Ga_n \cp \Qa'$ where $\Qa$ and $\Qa'$ are transitionless mope nets and the $\Ga_i$ are body nets. 
Since  $\Qa$ and $\Qa'$ are transitionless, $F_\bt(\Qa) = 0$ and $F_\bt(\Qa') = 0$ for all body types $\bt$. Therefore, functoriality of $F$, $F_{a,z}$, and the $F_\beta$ implies that it suffices to show that $F(\Qa) + F(\Qa') = F_{a,z}(\Qa) + F_{a,z}(\Qa')$ and that $F$ and $F_{a,z} + \sum_{\bt} d_\bt F_\bt$ agree on $\Ga_i$ for all $i$.

The proof of Lemma~\ref{lem:extract-net}  in fact shows that $\Qa = (X \to S \xleftarrow{1_S} S, 0_S)$ and $\Qa' = (S \xrightarrow{1_S} S \leftarrow Y, 0_S)$. 

Note that $F(\Qa \cp \Qa') = F(X \to S \leftarrow Y,  0_S)$. By Lemma~\ref{lem:functor-implies-condition}, this equals 
\[
    (F(\eta_{S}) - F(\eta_{X})) + (F(\epsilon_{S}) - F(\epsilon_{Y}))
\] which by definition of $a$ and $z$ equals $(a_{|S|} - a_{|X|}) + (z_{|S|} - z_{|Y|})$. Therefore, 
\begin{align*}
    F(\Qa) + F(\Qa') &= F(\Qa \cp \Qa') \\
    & = (a_{|S|} - a_{|X|}) + (z_{|S|} - z_{|Y|})\\
    & = F_{a,z}(\Qa) + F_{a,z}(\Qa').
\end{align*}

For each $i$, let $\tau_i$ be the single transition of $\Ga_i$ and let $\bt_i$ be the body type of $\tau_i$. Let $S_i$ be the species in the decoration of $\Ga_i$ which are neither  source nor target species of $\tau_i$. Then $\Ga_i = \Pa_{\bt_i} \oplus 1_{S_i}$.  In other words, since each $\Ga_i$ is body mope net, it is isomorphic to the monoidal product of an irreducible body net and an identity mope net. Recall that the canonical irreducible body nets $\Pa_\bt$ are defined in Definition~\ref{def:body-type}. By Lemma~\ref{lem:moo} any functor $\mope \to \bn$ agree on $\Ga_i$ and $\Pa_{\bt_i}$.

Since $F$ and $F_{\bt_i}$ are $0$ on the identity and $F_{\bt_i}(\Pa_{\bt_i}) = 1$ we have, 
\[
    F(\Ga_i) = F(\Pa_{\bt_i}) = F(\Pa_{\bt_i})F_{\bt_i}(\Pa_{\bt_i}) = 
    F(\Pa_{\bt_i})F_{\bt_i}(\Ga_i).
\]
For body types $\bt \neq \bt_i$, \[
    F_\bt(\Ga_i) = F_\bt(\Pa_{\bt_i}) = 0.
\]
Finally, the underlying cospan of $\Ga_i$ is the identity, so $F_{a,z}(\Ga_i) = 0$.  All together, this implies that $F$ and $F_{a,z} + \sum_{\bt} d_\bt F_\bt$ agree on $\Ga_i$.

\end{proof}

\section{Final Remarks}
\noindent
This paper presents an in-depth study of additive invariants for open Petri nets, conceptualized as a monoidal functor from a symmetric monoidal cospan category, $\openpetri$, to $\bn$. Our research makes significant strides in both the development and classification of all $\mathbb{N}$-valued additive invariants.

We have shown that all functors from $\openpetri$ to $\cc{B}\mathbb{N}$ are monoidal, and thus additive invariants. Our study characterizes these invariants for both general open nets in \textbf{OPetri} and the subcategory \textbf{MOPetri}, which comprises those open Petri nets whose underlying cospan has monomorphic legs. Further, we found that the additive invariants are completely characterized by their values on specific classes of Petri nets: for general open Petri nets, the generators are single-transition nets, whereas for \textbf{MOPetri}, they include both all single-transition and transitionless Petri nets. In exploring the broader applicability for our main classification theorems, we identify a special subclass of cancellative monoids, characterized by the unique property that the identity is the sole element with an inverse.

The classification of these additive invariants relies strongly on key decomposition lemmas for open Petri nets, which assert that any open Petri net can be canonically factorized, using tensor or composite operations, into single-transition and transitionless open Petri nets. These Lemmas are of independent interest given the significant literature on open Petri nets in applied category theory.

In this study, we focused on the invariants of open Petri nets with the codomain $\cc{B}\mathbb{N}$. The classification problem of functors out of $\openpetri$, however seems very challenging in general. Despite the progress made in this paper, several avenues remain open for future research, particularly in exploring invariants within different monoids. Investigating invariants beyond numerical values, such as $\mathbb{Z}/2$-valued of open Petri nets, raises further inquiry into whether such invariants could provide new insights not captured by $\mathbb{N}$-valued invariants, similar to how $\mathbb{Z}/2$-valued co/homology invariants in topology reveal aspects not visible in $\mathbb{Z}$-valued invariants.

Furthermore, our exploration of additive invariants confirms their role as compositional semantics for Petri nets. Mass action kinetics, while a highly descriptive invariant on open Petri nets, are computationally impractical for verification and specification.  Conversely the relative simplicity of the additive invariants presented in this paper, make them particularly practical for developing domain-specific Petri nets. An additive invariant can represent constraints either imposed by domain experts or discovered empirically through real-world applications. The compositional nature of these invariants is critical for constraint verification and for modularly constructing networks to fulfill specific constraint requirements. These features are well-suited for integration into the software packages \href{https://github.com/AlgebraicJulia/AlgebraicPetri.jl}{AlgebraicPetri.jl} and Catcolab~\cite{catcolab}. Although not currently implemented in full, we have developed a working prototype. A full implementation would involve (1) a decomposition for open Petri nets into atomic components and (2) a mechanism to define additive invariants by assigning values to each atomic net.

To sum up, this paper promote for a more algebraic and compositional understanding of Petri nets. It lays the groundwork of valuable theoretical insights with potential applications in computational biology, computer science, network theory, and systems design with a deeper understanding of complex systems.

\bibliographystyle{plainnat}
\bibliography{biblio}

\begin{thebibliography}{24}
\providecommand{\natexlab}[1]{#1}
\providecommand{\url}[1]{\texttt{#1}}
\expandafter\ifx\csname urlstyle\endcsname\relax
  \providecommand{\doi}[1]{doi: #1}\else
  \providecommand{\doi}{doi: \begingroup \urlstyle{rm}\Url}\fi

\bibitem[Aduddell et~al.(2024)Aduddell, Fairbanks, Kumar, Ocal, Patterson, and Shapiro]{motifs}
Rebekah Aduddell, James Fairbanks, Amit Kumar, Pablo~S. Ocal, Evan Patterson, and Brandon~T. Shapiro.
\newblock A compositional account of motifs, mechanisms, and dynamics in biochemical regulatory networks.
\newblock \emph{Compositionality}, Volume 6 (2024):\penalty0 2, May 2024.
\newblock ISSN 2631-4444.
\newblock \doi{10.32408/compositionality-6-2}.
\newblock URL \url{https://compositionality.episciences.org/13637}.

\bibitem[Baez and Biamonte(2018)]{baez2018quantum}
J.~C. Baez and J.~D. Biamonte.
\newblock \emph{Quantum Techniques in Stochastic Mechanics}.
\newblock World Scientific, 2018.
\newblock URL \url{https://doi.org/10.48550/arXiv.1209.3632}.

\bibitem[Baez and Pollard(2017)]{reaction-networks-pollard-baez}
J.~C. Baez and B.~S. Pollard.
\newblock A {C}ompositional {F}ramework for {R}eaction {N}etworks.
\newblock \emph{Reviews in Mathematical Physics}, 29\penalty0 (09):\penalty0 1750028, 2017.
\newblock URL \url{https://doi.org/10.1142/S0129055X17500283}.

\bibitem[Baldan et~al.(2005)Baldan, Corradini, Ehrig, and Heckel]{baldan_corradini_ehrig_heckel_2005}
P.~Baldan, A.~Corradini, H.~Ehrig, and R.~Heckel.
\newblock Compositional semantics for open {P}etri nets based on deterministic processes.
\newblock \emph{Mathematical Structures in Computer Science}, 15\penalty0 (1):\penalty0 1–35, 2005.
\newblock \doi{10.1017/S0960129504004311}.

\bibitem[Baldan et~al.(2001)Baldan, Corradini, Ehrig, and Heckel]{baldan2001compositional}
Paolo Baldan, Andrea Corradini, Hartmut Ehrig, and Reiko Heckel.
\newblock Compositional modeling of reactive systems using open nets.
\newblock In \emph{CONCUR 2001—Concurrency Theory: 12th International Conference Aalborg, Denmark, August 20--25, 2001 Proceedings}, volume 2154 of \emph{Lecture Notes in Computer Science}, pages 502--518. Springer, 2001.

\bibitem[Brown et~al.()Brown, Carlson, Lynch, and Patterson]{catcolab}
Kris Brown, Kevin Carlson, Owen Lynch, and Evan Patterson.
\newblock Catcolab software.
\newblock URL \url{https://catcolab.org/help/credits}.

\bibitem[Fong(2016)]{fong2016algebra}
B.~Fong.
\newblock \emph{{T}he {A}lgebra of {O}pen and {I}nterconnected {S}ystems}.
\newblock Ph.D. thesis, Computer Science Department, University of Oxford, 2016.

\bibitem[Fong and Spivak(2019)]{fong2019hypergraph}
B.~Fong and D.~I. Spivak.
\newblock {H}ypergraph {C}ategories.
\newblock \emph{Journal of Pure and Applied Algebra}, 223\penalty0 (11):\penalty0 4746--4777, 2019.

\bibitem[Gadducci and Heckel(1997)]{gadducci1997inductive}
Fabio Gadducci and Reiko Heckel.
\newblock An inductive view of graph transformation.
\newblock In \emph{International Workshop on Algebraic Development Techniques}, volume 1376 of \emph{Lecture Notes in Computer Science}, pages 223--237. Springer, 1997.

\bibitem[Haas(2002)]{haas2006stochastic}
P.~J. Haas.
\newblock \emph{Stochastic Petri Nets: Modelling, Stability, Simulation}.
\newblock Springer New York, NY, 2002.
\newblock URL \url{https://doi.org/10.1007/b97265}.

\bibitem[Jensen and Kristensen(2009)]{jensen2009coloured}
K.~Jensen and L.~M. Kristensen.
\newblock \emph{Coloured Petri Nets: Modelling and Validation of Concurrent Systems}.
\newblock Springer Science \& Business Media, 2009.

\bibitem[Johnson and Yau(2021)]{johnson20212}
N.~Johnson and D.~Yau.
\newblock \emph{2-Dimensional Categories}.
\newblock Oxford University Press, Oxford, 2021.
\newblock URL \url{https://doi.org/10.1093/oso/9780198871378.001.0001}.

\bibitem[Keser{\H{u}} et~al.(2014)Keser{\H{u}}, So{\'o}s, and Kappe]{keserHu2014anthropogenic}
G.~M. Keser{\H{u}}, T.~So{\'o}s, and C.~O. Kappe.
\newblock {A}nthropogenic {R}eaction {P}arameters--{T}he {M}issing {L}ink {B}etween {C}hemical {I}ntuition and the {A}vailable {C}hemical {S}pace.
\newblock \emph{Chemical Society Reviews}, 43\penalty0 (15):\penalty0 5387--5399, 2014.
\newblock URL \url{https://doi.org/10.1039/C3CS60423C}.

\bibitem[Koch(2010)]{koch2010petri}
I.~Koch.
\newblock {P}etri {N}ets--{A} {M}athematical {F}ormalism to {A}nalyze {C}hemical {R}eaction {N}etworks.
\newblock \emph{Molecular Informatics}, 29\penalty0 (12):\penalty0 838--843, 2010.
\newblock URL \url{https://doi.org/10.1002/minf.201000086}.

\bibitem[Koch et~al.(2010)Koch, Reisig, and Schreiber]{koch2010modeling}
I.~Koch, W.~Reisig, and F.~Schreiber.
\newblock \emph{Modeling in Systems Biology: the Petri Net Approach}, volume~16.
\newblock Springer science \& business media, 2010.
\newblock URL \url{https://doi.org/10.1007/978-1-84996-474-6}.

\bibitem[Lawson et~al.(2014)Lawson, Swienty-Busch, G{\'e}oui, and Evans]{lawson2014making}
A.~J. Lawson, J.~Swienty-Busch, T.~G{\'e}oui, and D.~Evans.
\newblock {T}he {M}aking of {R}eaxys—{T}owards {U}nobstructed {A}ccess to {R}elevant {C}hemistry {I}nformation.
\newblock In \emph{{T}he {F}uture of the {H}istory of {C}hemical {I}nformation}, pages 127--148. ACS Publications, 2014.
\newblock URL \url{https://doi.org/10.1021/bk-2014-1164.ch008}.

\bibitem[Marsan et~al.(1998)Marsan, Balbo, Conte, Donatelli, and Franceschinis]{marsan1998modelling}
M.~A. Marsan, G.~Balbo, G.~Conte, S.~Donatelli, and G.~Franceschinis.
\newblock {M}odelling with {G}eneralized {S}tochastic {P}etri {N}ets.
\newblock \emph{ACM SIGMETRICS {P}erformance {E}valuation {R}eview}, 26\penalty0 (2):\penalty0 2, 1998.

\bibitem[Nielsen et~al.(1995)Nielsen, Priese, and Sassone]{nielsen1995characterizing}
Mogens Nielsen, Lutz Priese, and Vladimiro Sassone.
\newblock Characterizing behavioural congruences for {P}etri nets.
\newblock In \emph{CONCUR'95: Concurrency Theory: 6th International Conference Philadelphia, PA, USA, August 21--24, 1995 Proceedings}, volume 962 of \emph{Lecture Notes in Computer Science}, pages 175--189. Springer, 1995.

\bibitem[Peterson(1981)]{petri-net-theory-book}
J.~L. Peterson.
\newblock \emph{Petri Net Theory and the Modeling of Systems}.
\newblock Prentice Hall PTR, USA, 1981.
\newblock ISBN 0136619835.

\bibitem[Petri and Reisig(2008)]{petri-petrinet}
C.~A. Petri and W.~Reisig.
\newblock {P}etri {N}et.
\newblock \emph{Scholarpedia}, 2008.
\newblock URL \url{https://doi.org/10.4249/scholarpedia.6477}.

\bibitem[Reed et~al.(2003)Reed, Vo, Schilling, and Palsson]{reedexpanded}
J.L. Reed, T.D. Vo, C.H. Schilling, and B.O. Palsson.
\newblock {A}n {E}xpanded {G}enome-{S}cale {M}odel of {E}scherichia {C}oli {K}-12 (ijr904 gsm/gpr).
\newblock \emph{Genome Biol}, 4, 2003.
\newblock URL \url{https://doi.org/10.1186/gb-2003-4-9-r54}.

\bibitem[Tyson and Nov{\'a}k(2010)]{tyson2010functional}
J.~J. Tyson and B.~Nov{\'a}k.
\newblock {F}unctional {M}otifs in {B}iochemical {R}eaction {N}etworks.
\newblock \emph{Annual {R}eview of {P}hysical {C}hemistry}, 61:\penalty0 219--240, 2010.
\newblock URL \url{https://doi.org/10.1146/annurev.physchem.012809.103457}.

\bibitem[Van~Brakel(2012)]{van2012substances}
J.~Van~Brakel.
\newblock {S}ubstances: {T}he {O}ntology of {C}hemistry.
\newblock In \emph{{P}hilosophy of {C}hemistry}, pages 191--229. Elsevier, Amsterdam, 2012.
\newblock URL \url{https://doi.org/10.1016/B978-0-444-51675-6.50018-9}.

\bibitem[Wilkinson(2018)]{wilkinson2018stochastic}
D.~J. Wilkinson.
\newblock \emph{Stochastic Modelling for Systems Biology}.
\newblock Chapman and Hall/CRC, NY, 2018.
\newblock URL \url{https://doi.org/10.1201/9781351000918}.

\end{thebibliography}

\end{document}